\documentclass[a4paper,10pt]{amsart}

\usepackage[latin1]{inputenc}
\usepackage{amsmath,amsthm,amssymb}
\usepackage{epsfig}
\usepackage{caption}

\usepackage{tikz}
\usepackage{xifthen}
\usetikzlibrary{calc,decorations.markings,arrows}
\usepackage[all]{xy}
\usepackage{graphicx}

\newcommand{\al}{\alpha}
\newcommand{\eps}{\varepsilon}

\newcommand{\vph}{\varphi}

\newcommand{\fl}{\rightarrow}
\newcommand{\gfl}{\longrightarrow}

\newcommand{\dr}{\ar@{->}[r]}
\newcommand{\dri}{\ar@{>->}[r]}
\newcommand{\drp}{\ar@{-->}[r]}
\newcommand{\dre}{\ar@{->>}[r]}
\newcommand{\dreg}{\ar@{=}[r]}
\newcommand{\drm}{\ar@{^{(}->}[r]}
\newcommand{\ddr}{\ar@{->}[rr]}
\newcommand{\ddre}{\ar@{->>}[rr]}
\newcommand{\ddreg}{\ar@{=}[rr]}
\newcommand{\ha}{\ar@{->}[u]}
\newcommand{\hae}{\ar@{->>}[u]}
\newcommand{\hap}{\ar@{-->}[u]}
\newcommand{\ham}{\ar@{^{(}->}[u]}
\newcommand{\hham}{\ar@{^{(}->}[uu]}
\newcommand{\hag}{\ar@{->}[ul]}
\newcommand{\hagm}{\ar@{^{(}->}[ul]}
\newcommand{\hagp}{\ar@{-->}[ul]}
\newcommand{\hdr}{\ar@{->}[ur]}
\newcommand{\hdrm}{\ar@{^{(}->}[ur]}
\newcommand{\hdri}{\ar@{>->}[ur]}
\newcommand{\hdre}{\ar@{->>}[ur]}
\newcommand{\hdrp}{\ar@{-->}[ur]}
\newcommand{\bas}{\ar@{->}[d]}
\newcommand{\bbas}{\ar@{->}[dd]}
\newcommand{\basm}{\ar@{^{(}->}[d]}
\newcommand{\basi}{\ar@{>->}[d]}
\newcommand{\base}{\ar@{->>}[d]}
\newcommand{\basp}{\ar@{-->}[d]}
\newcommand{\baseg}{\ar@{=}[d]}
\newcommand{\bbaseg}{\ar@{=}[dd]}
\newcommand{\bdr}{\ar@{->}[dr]}
\newcommand{\bdre}{\ar@{->>}[dr]}
\newcommand{\bbdr}{\ar@{->}[ddr]}
\newcommand{\bddr}{\ar@{->}[drr]}
\newcommand{\bg}{\ar@{->}[dl]}
\newcommand{\bgm}{\ar@{^{(}->}[dl]}
\newcommand{\bgp}{\ar@{-->}[dl]}
\newcommand{\bggp}{\ar@{-->}[dll]}
\newcommand{\bbgp}{\ar@{-->}[ddl]}
\newcommand{\bdro}{\ar@/_.5pc/@{->}[dr]^0}
\newcommand{\bgex}{\ar@/^.5pc/@{-->}[dl]}

\newcommand{\ac}{\mathcal{A}}
\newcommand{\bc}{\mathcal{B}}
\newcommand{\cat}{\mathcal{C}}

\newcommand{\ec}{\mathcal{E}}

\renewcommand{\sc}{\mathcal{S}}
\newcommand{\tc}{\mathcal{T}}

\newcommand{\modl}{\operatorname{mod}\Lambda}

\newcommand{\shift}{\Sigma}
\newcommand{\susm}{\Sigma^{-1}}

\newcommand{\ext}{\operatorname{Ext}}
\newcommand{\homph}{\operatorname{Hom}}

\theoremstyle{plain}
\newtheorem{lemma}{Lemma}[section]
\newtheorem{theo}[lemma]{Theorem}
\newtheorem{cor}[lemma]{Corollary}

\newtheorem{prop}[lemma]{Proposition}

\newtheorem*{cornonumber}{Corollary} 

\newtheorem{thmx}{Theorem}

\numberwithin{equation}{section}

\theoremstyle{definition}
\newtheorem{defi}[lemma]{Definition}
\newtheorem{rk}[lemma]{Remark}

\newcommand{\End}{\operatorname{End}}

\newcommand{\had}{\ar@[]}

\newcommand{\add}{\operatorname{add}}

\newcommand{\tbar}{\overline{T}}
\newcommand{\ubar}{\overline{U}}
\newcommand{\vbar}{\overline{V}}
\newcommand{\ct}{\cat(T)}
\newcommand{\cbart}{\overline{\cat}(T)}
\newcommand{\qbar}{\overline{Q}}

\newcommand{\catt}{\cbart/(\shift T')}

\newcommand{\fbar}{\overline{F}}
\newcommand{\lbar}{\overline{\Lambda}}
\newcommand{\modlbar}{\operatorname{mod}\lbar}

\newcommand{\morone}{\mathcal{R}}
\newcommand{\mortwo}{\morone^{\ast}}

\title{Nearly Morita equivalences and rigid objects}

\author{Robert Marsh}
\address{School of Mathematics, University of Leeds, Leeds, LS2 9JT, UK}
\email{marsh@maths.leeds.ac.uk}

\author{Yann Palu}
\address{LAMFA, Facult\'e des sciences, 33, rue Saint-Leu, 80039 Amiens Cedex 1}
\email{yann.palu@u-picardie.fr}

\begin{document}

\begin{abstract}
If $T$ and $T'$ are two cluster-tilting objects of an acyclic cluster category
related by a mutation, their endomorphism algebras are nearly-Morita equivalent~\cite{BMR-CTA}, i.e.\ 
their module categories are equivalent ``up to a simple module''.
This result has been generalised by D. Yang, using a result of P-G. Plamondon,
to any simple mutation of maximal rigid objects in a 2-Calabi--Yau triangulated category.
In this paper, we investigate the more general case of any mutation of a (non-necessarily maximal)
rigid object in a triangulated category with a Serre functor.
In that setup, the endomorphism algebras might not be nearly-Morita equivalent and we obtain a weaker property that we call pseudo-Morita equivalence.
Inspired by~\cite{BMloc2,BMloc1}, we also describe our result in terms of localisations.
\end{abstract}

\thanks{This work was supported by the Engineering and Physical Sciences Research Council
[grant number EP/G007497/1] and the Institute for Mathematical Research (FIM) at the ETH Z\"{u}rich.}

\maketitle

\tableofcontents

\section*{Introduction and main results}
In this paper, our aim is to prove a weak form of nearly-Morita equivalence for mutations of (non-maximal) rigid objects in triangulated categories.
Before recalling the case of cluster-tilting objects~\cite{BMR-CTA}, we first give an example.

Let $Q$ be a linear orientation of the Dynkin
diagram of type $A_3$. The Auslander--Reiten quiver
of the acyclic cluster category $\cat_Q$, defined in~\cite{BMRRT}, is as follows:
$$
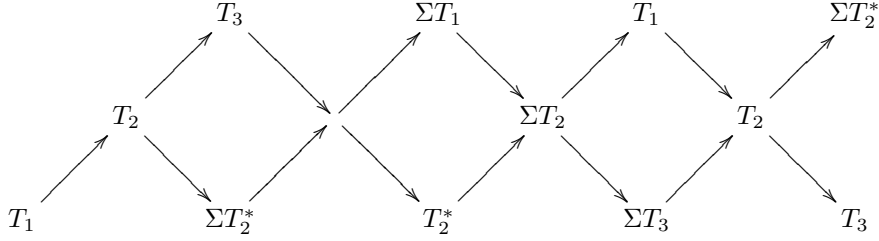
\xymatrix@=0.7cm@!@-.5pc{
&&  T_3 \ar[dr]
&&   \shift T_1    \ar[dr]
&&   T_1 \ar[dr]
&&  \shift T_2^\ast   \\
&   T_2 \ar[dr] \ar[ur]
&&   \ar[dr] \ar[ur]
&&   \shift T_2 \ar[dr] \ar[ur]
&&   T_2 \ar[dr] \ar[ur]
&       \\
    T_1 \ar[ur]
&&  \shift T_2^\ast \ar[ur]
&&  T_2^\ast \ar[ur]
&&  \shift T_3\ar[ur]
&&  T_3
}
$$
The object $T=T_1\oplus T_2\oplus T_3$ is cluster-tilting. Its mutation
at $T_2$ is the cluster-tilting object $T'=T_1\oplus T_2^\ast\oplus T_3$.
We write $\Gamma$ for the cluster-tilted algebra $\operatorname{End}_\cat(T)^\text{op}$
and $\Gamma'$ for $\operatorname{End}_\cat(T')^\text{op}$.
Then the two algebas $\Gamma$ and $\Gamma'$ are related as follows.

On the one hand, the functor $\cat(T,-)$ induces an equivalence of categories
$\cat/(\shift T) \simeq \operatorname{mod} \Gamma$, where $\operatorname{mod} \Gamma$
is the category of finitely generated left modules,
and the Auslander--Reiten quiver of $\operatorname{mod} \Gamma$
is thus:
$$
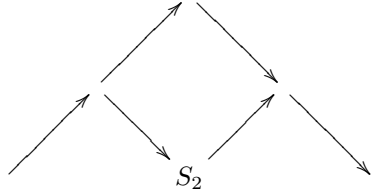
\xymatrix@!@-.5pc{
&& \ar[dr]
&&                   \\
&  \ar[dr] \ar[ur]
&& \ar[dr]
&                    \\
   \ar[ur]
&& S_2 \ar[ur]
&&
}
$$
where $S_2=\cat(T,\shift T_2^\ast)$
is the simple top of the projective indecomposable
$\cat(T,T_2)$.

On the other hand, the functor
$\cat(T',-)$ induces an equivalence of categories
$\cat/(\shift T') \simeq \operatorname{mod} \Gamma'$
and the Auslander--Reiten quiver of $\operatorname{mod} \Gamma'$
is thus:
$$
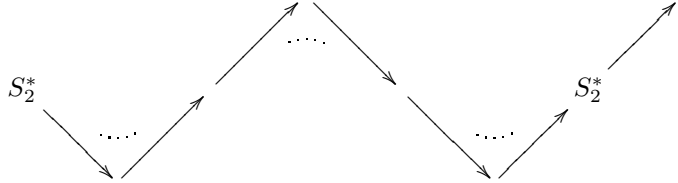
\xymatrix@!@-.5pc{
&&&   \ar[dr]_{\phantom{t}}="t"
&&&&                   \\
      S_2^\ast \ar[dr]^{\phantom{a}}="a"
&&    \ar[ur]_{\phantom{s}}="s"
&&    \ar[dr]^{\phantom{c}}="c"
&&    S_2^\ast \ar[ur]
&                      \\
&     \ar[ur]^{\phantom{b}}="b"
&&&&  \ar[ur]^{\phantom{d}}="d"
&&
\ar@{..}@/_/"a";"b"
\ar@{..}@/^/"s";"t"
\ar@{..}@/_/"c";"d"
}
$$
where $S_2^\ast=\cat(T',\shift T_2)$
is the simple top of the projective indecomposable
$\cat(T',T_2^\ast)$,
where the two arrows starting at $S_2^\ast$ are identified,
and where dots indicate zero relations.

The two Auslander--Reiten quivers are not isomorphic,
therefore $\Gamma$ and $\Gamma'$ are not Morita equivalent.
But they are not very far from being so: The difference
in the Auslander--Reiten quivers comes from the simples $S_2$ and $S_2^\ast$.

The common Auslander-Reiten quiver of the
categories $\operatorname{mod} \Gamma/(\add S_2)$
and $\operatorname{mod} \Gamma/(\add S_2^{\ast})$
is thus:
$$
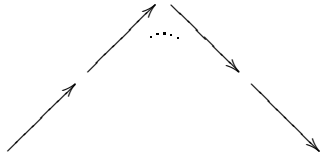
\xymatrix@!{
&& \ar[dr]_{\phantom{t}}="t"
&& \\
&  \ar[ur]_{\phantom{s}}="s"
&& \ar[dr]
&  \\
   \ar[ur]
&&&&
\ar@{..}@/^/"s";"t"
}
$$
This phenomenon, proved in \cite{BMR-CTA},
has been called ``nearly Morita equivalence'' by C. M. 
Ringel. Let us state the precise result.

Let $Q$ be an acyclic quiver, and let $T$ be a cluster-tilting object in the cluster category $\cat_Q$.
Let $T'=T/T_k\oplus T_k^\ast$ be the mutation of $T$
at an indecomposable summand $T_k$; then $T'$
is also a cluster-tilting object.
Let $\Gamma$ (respectively, $\Gamma'$)
be the cluster-tilted algebra $\operatorname{End}_{\cat_Q}(T)^\text{op}$
(respectively, $\operatorname{End}_{\cat_Q}(T')^\text{op}$)
and $S_k$ (respectively, $S_k^\ast$) be the simple top of the
projective indecomposable $\Gamma$-module $\cat_Q(T,T_k)$
(respectively, the simple top of the
$\Gamma'$-module $\cat_Q(T',T_k^\ast)$).

Then, by a result of~\cite{BMR-CTA}, the
categories $\operatorname{mod} \Gamma / \add S_k$ and $\operatorname{mod}\Gamma' / \add S_k^\ast$
are equivalent. By~\cite[Corollary 4.3]{Yang-ClusterTube}, nearly-Morita equivalence, in the more general setup of simple, $2$-periodic mutations of rigid objects (or rigid, Krull--Schmidt subcategories)
in any triangulated category, follows from~\cite[Proposition 2.7]{Plamondon-CC}.

Our main aim in this paper is to prove an analoguous result for any mutation of (non-maximal) rigid objects.
Before explaining our results, let us have a look at an example which shows that one cannot expect these mutations to induce a nearly-Morita equivalence in general.

Let $T=T_1\oplus T_2\oplus T_3$ be the rigid
object of the acyclic cluster category $\cat = \cat_{A_4}$
given by:
$$
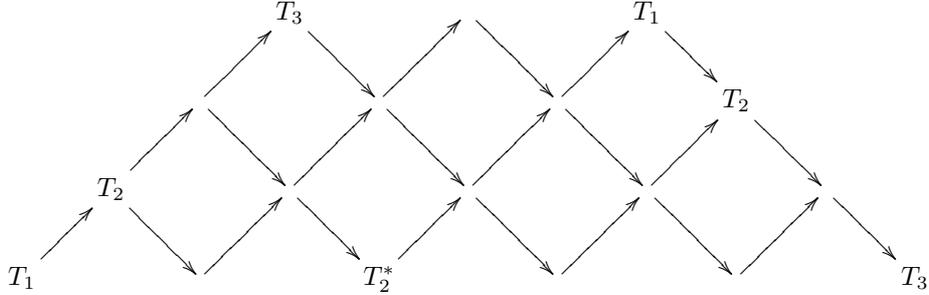
\xymatrix@!@-.7pc{
&&&  T_3 \ar[dr]
&&       \ar[dr]
&&   T_1 \ar[dr]
&&&        \\
&&   \ar[dr] \ar[ur]
&&   \ar[dr] \ar[ur]
&&   \ar[dr] \ar[ur]
&&   T_2 \ar[dr]
&&           \\
&    T_2 \ar[dr] \ar[ur]
&&  \ar[dr] \ar[ur]
&&  \ar[dr] \ar[ur]
&&  \ar[dr] \ar[ur]
&&  \ar[dr]
&            \\
     T_1 \ar[ur]
&&       \ar[ur]
&&   T_2^\ast \ar[ur]
&&       \ar[ur]
&&       \ar[ur]
&&   T_3
}
$$
and let $T'=T_1\oplus T_2^\ast\oplus T_3$ be
the rigid object obtained by mutating $T$ at
the summand $T_2$. This means that $\shift T_2^\ast$
is the cone of a minimal right $\add T/T_2$-approximation of $T_2$.
In the example, there is a triangle
$T_2^\ast\fl T_1\fl T_2\fl \shift T_2^\ast$.
Let $\Lambda$ (respectively, $\Lambda'$) be the algebra
$\operatorname{End}_\cat(T)^\text{op}$ (respectively, $\operatorname{End}_\cat(T')^\text{op}$).
Using results in \cite{BMloc2}, \cite{BMloc1} (see also~\cite{KR1}), we can easily compute the
AR quivers of $\operatorname{mod}\Lambda$ and $\operatorname{mod}\Lambda'$:
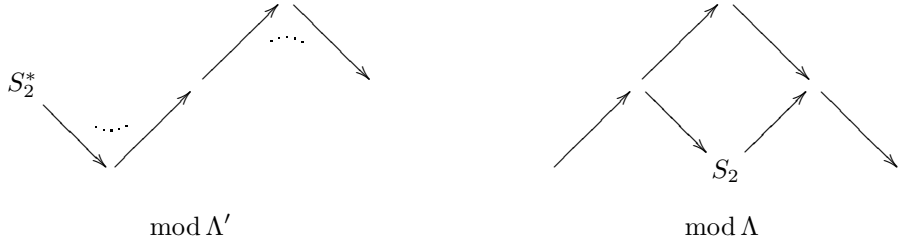
$$
\xymatrix@!@-.7pc{
&&& \ar[dr]_{\phantom{t}}="t" &  && && \ar[dr] && \\
S_2^\ast \ar[dr]^{\phantom{a}}="a" && \ar[ur]_{\phantom{s}}="s" &&  && & \ar[dr] \ar[ur] && \ar[dr] & \\
& \ar[ur]^{\phantom{b}}="b" &&&  && \ar[ur] && S_2 \ar[ur] && 
\ar@{.}@/_/"a";"b" \ar@{..}@/^/"s";"t"
}
$$
$$
\operatorname{mod} \Lambda' \hspace{14pc} \operatorname{mod} \Lambda
$$

The algebras $\Lambda$ and $\Lambda'$ are not nearly-Morita equivalent. On factoring out by $S_2$
(respectively, $S_2^{\ast}$), we obtain the following
Auslander-Reiten quivers:
$$
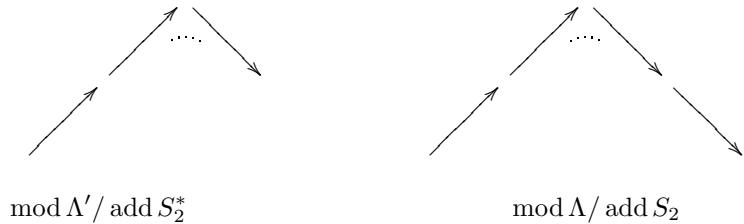
\xymatrix{
&&& \ar[dr]_{\phantom{t}}="t" &  && && \ar[dr]_{\phantom{s}}="b" && \\
&&  \ar[ur]_{\phantom{s}}="s" &&  && & \ar[ur]_{\phantom{t}}="a" && \ar[dr] & \\
&   \ar[ur] &&&  && \ar[ur] &&  && 
\ar@{..}@/^/"a";"b" \ar@{..}@/^/"s";"t"
}
$$
$$
\operatorname{mod} \Lambda' / \operatorname{add}S_2^\ast \hspace{10pc} \operatorname{mod} \Lambda/\operatorname{add}S_2
$$
However, these algebras are not very far from being nearly-Morita equivalent.
Indeed, the Auslander--Reiten quivers differ by only one arrow. The corresponding morphism can be characterised in $\operatorname{mod} \Lambda$ as
being surjective with kernel in the subcategory $\add S_2$.

Let $\cat$ be an acyclic cluster category, and let $T$ be a rigid object in $\cat$.
Let $T'=T/T_k\oplus T_k^\ast$ be the mutation of $T$
at the summand $T_k$.
Let $\Lambda$ (respectively, $\Lambda'$)
be the algebra $\operatorname{End}_\cat(T)^\text{op}$
(respectively, $\operatorname{End}_\cat(T')^\text{op}$),
and let $S_k$ (respectively, $S_k^\ast$) be the simple top of the projective indecomposable $\Lambda$-module $\cat(T,T_k)$ (respectively, the $\Lambda'$-module $\cat(T',T_k^\ast)$).

As suggested by the example above, 
let us consider the class $\morone$ of epimorphisms in $\operatorname{mod}\Lambda$ with kernels in $\add S_k$, and the class $\mortwo$ of monomorphisms in
$\operatorname{mod}\Lambda'$ with cokernels in
$\add S_k^{\ast}$.

\begin{thmx} \label{theoremA}
There is an equivalence of categories:
$$(\operatorname{mod}\Lambda)_{\morone}\simeq (\operatorname{mod}\Lambda')_{\mortwo}.$$
\end{thmx}

This result is not completely satisfactory since it does not resemble nearly-Morita equivalence. The following remark will help in restating the Theorem in a form which looks more like nearly-Morita equivalence.

Let $M\in\operatorname{mod}\Lambda$.
If there is a short exact sequence
$0\fl S_k \fl L \stackrel{f}{\fl} M \fl 0$,
the morphism $f$ belongs to $\morone$. Therefore
the objects $L$ and $M$ become isomorphic in the localisation
$(\operatorname{mod}\Lambda)_{\morone}$.
This suggests that the objects having non-split
extensions with $S_k$ can be removed from $\operatorname{mod}\Lambda$
without changing the localisation.
We thus define $\ec$ to be the full subcategory of $\operatorname{mod}\Lambda$
whose objects $M$ satisfy $\ext^1_\Lambda(M,S_k) = 0$.
Dually, let $\ec'$ be the full subcategory of $\operatorname{mod}\Lambda'$
whose objects $N$ satisfy $\ext^1_{\Lambda'}(S_k^\ast,N) = 0$.

Note that $\ec$ and $\ec'$ are extension-closed in $\operatorname{mod}\Lambda$ (respectively,
$\operatorname{mod}\Lambda'$) and are thus exact categories.

\begin{thmx} \label{theoremB}
There is an equivalence of categories:
$$(\operatorname{mod}\Lambda)_{\morone}\simeq \ec / \add S_k.$$ Dually, there is an equivalence of categories:
$$(\operatorname{mod}\Lambda')_{\mortwo}\simeq \ec' / \add S_k^\ast.$$
\end{thmx}

Combining the two theorems gives the following.

\begin{cornonumber}
There is an equivalence of categories:
$$\ec / \add S_k \simeq \ec' / \add S_k^\ast.$$
\end{cornonumber}

This resembles nearly-Morita equivalence except that,
unlike in the cluster-tilting case, one has to restrict to an exact subcategory before killing the simple.

Unfortunately, these statements do not specialise to a nearly-Morita equivalence in the cluster-tilting case:
In the setup of~\cite{BMR-CTA}, we obtain
a weaker statement.

The proofs of Theorems A and B are in Subsection~\ref{ssection: localisation} (but note that the proofs
appear in reverse order to the above).
In fact, we will prove more general results than those
mentioned above.
First, we only assume the triangulated category $\cat$ to be Krull--Schimdt, with a Serre functor.
Second, we allow mutations at non-indecomposable summands.
Our results hold, in particular, in any triangulated category in the following list (whose items overlap):
\begin{itemize}
 \item Hom-finite generalised higher cluster categories (\cite{Amiot-ClusterCategories}, \cite{Guo-HigherClusterCategories});
 \item stable categories of maximal Cohen--Macaulay modules over an odd dimensional
isolated hypersurface singularity (\cite{BIKR});
 \item cluster tubes (\cite{BKL}, \cite{BMV}...);
 \item (higher) cluster categories of type $A_\infty$ (\cite{HJ-Ainfinity}, \cite{HJ-HigherClusterCategories});
 \item the triangulated orbit categories listed in~\cite{Amiot-TriangulatedCategories};
 \item stable categories constructed from preprojective algebras in~\cite{GLS-ClusterAlgebraStructures}...
\end{itemize}

\section*{Acknowledgements}
The second-named author would like to thank the algebra team of the university of Leeds
for a pleasant atmosphere when he was a postdoc there.
Both authors are indebted to Apostolos Beligiannis for a preliminary version of~\cite{Beligiannis} which inspired the proofs of section~\ref{ssection: adjunctions}, and to
Karin Baur who hosted their visits to the Institute for Mathematical Research (FIM) at the ETH Z\"{u}rich, where this project was initiated.

\section{Setup and notation}
\label{s:setup}

We fix a field $k$, and a Krull--Schmidt, $k$-linear, Hom-finite, triangulated category $\cat$,
with suspension functor $\shift$.
An object $X$ in $\cat$ is called \emph{rigid} if
$\ext^1_\cat(X,X) = 0$, where we write $\ext^1_\cat(X,Y)$ for $\cat(X,\shift Y)$. We write $X^{\perp}$ for the
right Hom-perp of $X$, i.e.\ the subcategory of $\cat$
on objects $Y$ such that $\cat(X,Y)=0$. Note that this
notation differs from that used in~\cite{BMloc1}, which
we often cite, but here the Hom-perpendicular categories
play a key role so we use a different notation.

Let $T\in\cat$ be a basic rigid object. Let $R$ be a direct summand of $T$ and
write $T = \tbar\oplus R$. 
Let $T'$ be the rigid object obtained from $T$
by replacing $R$ by the negative shift $R^{\ast}$
of the cone
of a minimal right $\add\tbar$-approximation of $R$. We have a triangle
$R^\ast \fl B \fl R \fl \shift R^\ast$, with
$B\in\add\tbar$, $B\fl R$ a minimal right $\add\tbar$-approximation, and $T' = \tbar\oplus R^\ast$.
By~\cite[Lemma 6.7]{BMRRT}, $\Sigma R^*\in \tbar^{\perp}$
and $R^*$ is rigid, so that $T'$ is again rigid. By~\cite[Proposition 2.6(1)]{IY} and~\cite[Lemma 6.5]{BMRRT},
$R$ and $R^*$ are basic and have the same number
of indecomposable direct summands.
We keep these assumptions throughout the paper.

In some statements, we will assume additionally that
$\cat$ has a Serre functor.

We also need some more notation.
If $X$ is an object in $\cat$, we write 
$(X)$ for the ideal of morphisms factoring through the 
additive subcategory $\add X$ generated by $X$.
All modules considered are left modules.

We denote by $\ct$ the full subcategory
of $\cat$ whose objects are the cones of morphisms
$T_1\rightarrow T_0$, where $T_0,T_1\in \add T$,
and by $\cbart$ the full subcategory of $\cat$ whose
objects are the cones of morphisms $\overline{T_1}\rightarrow T_0$, where $T_0\in \add T$ and $\overline{T_1}\in \add \overline{T}$.

More generally, for any two full subcategories
$\ac$ and $\bc$ of $\cat$, we use the notation
$\ac\ast\bc$ for the full subcategory whose objects $X$
are extensions of an object in $\bc$ by an object in $\ac$ (i.e. $X$ appears in a triangle $A \fl X \fl B \fl \shift A$ with $A\in \ac$ and $B\in\bc$).
It follows from the octahedral axiom that the operation
$\ast$ is associative.
By abuse of notation, if $A,B$ are objects in $\cat$,
we will write $A\ast B$ for $\add A \ast \add B$.

Thus one could also define $\ct$ and $\cbart$ by:
$\ct = T\ast\shift T$ and $\cbart = T\ast\shift\tbar$.

\vspace{.2cm}
\emph{Remark}: Our results hold in the more general setup of rigid subcategories:
replace $\add T$ by a rigid subcategory $\tc$,
with the following additional assumptions: $\tc$ is contravariantly finite,
$\overline{\tc}$ is functorially finite and $\tc'$ is covariantly finite.
This requires changing the functors
of the form $\cat(T,-)$ taking values in the category
$\operatorname{mod}\operatorname{End}_\cat(T)^\text{op}$
into functors of the form $\cat(?,-)|_{\tc}$, taking values in
$\operatorname{mod}\tc$, and all references to~\cite{BMloc1} by references
to~\cite{Beligiannis}.

\section{Pseudo-Morita equivalence}

\subsection{Adjunctions}\label{ssection: adjunctions}

The methods used in this subsection are inspired by~\cite{Beligiannis,BMloc1,BMloc2}, and much resemble
results in~\cite[Section 3]{Nakaoka-Twin}.
Indeed, \cite[Corollary 3.8]{Nakaoka-Twin}
applied to the twin cotorsion pair
$(\shift\tbar,\tbar^\perp), (\shift T', T'^\perp)$
(where we use the notation from Subsection~\ref{ssection: main result})
gives the existence of a right adjoint to the fully faithful functor
$\cbart/(\shift T') \gfl \cat / (\shift T')$ from which it is possible
to deduce our Proposition~\ref{proposition: right adjoint}.
For convenience of the reader, we nonetheless include a complete proof.

The subcategory $\ct$ is known to be contravariantly 
finite, by~\cite[Lemmas 3.3 and 3.6]{BMloc1}. An 
analogous proof gives Lemma~\ref{lemma: cbart 
contravariantly finite} below. We first need a
definition.

\begin{defi}
Let $\sc$ be the set of morphisms $X\stackrel{f}{\gfl}Y$ in $\cat$
such that for any triangle $Z \fl X \stackrel{f}{\fl} Y \stackrel{g}{\fl}\shift Z$,
we have $Z\in\tbar^\perp$ and $g\in (T^\perp)$.
\end{defi}

\begin{lemma}
\label{lemma: cbart contravariantly finite}
\begin{itemize}
\item[(a)] Let $X'\stackrel{s}{\fl} X$ be a morphism in $\sc$ with
$X'\in \cbart$. Then $s$ is a right $\cbart$
approximation of $X$.
\item[(b)]
Each object $X$ in $\cat$ has a right $\cbart$-approximation $R_0X\stackrel{\eta_X}{\fl} X$
lying in $\sc$.
\item[(c)]
The category $\cbart$ is a contravariantly finite
subcategory of $\cat$.
\end{itemize}
\end{lemma}

\begin{proof}
Suppose that $X' \stackrel{s}{\fl} X$ is a
morphism in $\sc$ with $X'\in \cbart$.
Thus, we may complete $s$ to a triangle:
$$X' \stackrel{s}{\fl} X \stackrel{g}{\fl}
\shift Z \fl \shift X'$$
where $g$ factors through $T^{\perp}$ and
$\shift Z$ lies in $(\shift T)^{\perp}$

Since $X'\in \cbart$, there is a triangle
$U_0 \stackrel{p}{\fl} X' \fl \shift \ubar_1 \fl \shift U_0$,
with $U_0\in \add T$ and $\ubar_1 \in \add \tbar$.
Let $X'\stackrel{u}{\fl} X$ be an arbitrary morphism
in $\cat$.
Since $g$ factors through $T^{\perp}$ and $U_0\in \add T$, we have $gup=0$ and therefore have the following commutative diagram whose rows are triangles:
$$
\xymatrix{
U_0 \dr^p \basp_w & X'\bas^u \dr^\eta \bgp_{u'}
& \shift \ubar_1 \basp^v \dr & \shift U_0  \basp^{\shift w} \\
X' \dr_{s} & X \dr_g & \shift Z \dr & \shift X'
}
$$
Moreover, $\shift Z$ lies in $(\shift \tbar)^\perp$
and $\shift \ubar_1$ is in $\add \shift\tbar$, so
the composition $gu = v\eta$ is zero. Thus, there is a
morphism $u'$ such that $u=su'$. Part (a) is
shown.

For part (b), let $X\in\cat$. Let $T_0^X \fl X$ be a minimal right $\add T$-approximation of $X$.
Complete it to a triangle
$Y \fl T_0^X \fl X \fl \shift Y$.
Let $\tbar_1^Y \fl Y$ be a minimal right
$\add \tbar$-approximation of $Y$.
Applying the octahedral axiom, we obtain the
following diagram:
$$
\xymatrix{
\tbar_1^Y \dreg \bas & \tbar_1^Y \bas & & \\
Y \bas \dr & T_0^X \bas \dr & X \baseg \dr & \shift Y \bas \\
Z \bas \dr^f & R_0 X \bas \dr^{\eta_X} & X \dr^g & \shift Z \\
\shift \tbar_1^Y \dreg & \shift \tbar_1^Y 
}
$$
Applying the functors $\cat(T,-)$ and $\cat(\tbar,-)$
to the triangles above shows that
$\shift Y \in T^\perp$ and $Z\in\tbar^\perp$.
Note that $R_0 X \in \cbart$.
Then, by part (a), $\eta_X$ is a right $\cbart$-approximation of $X$, and part (b) is shown.
Part (c) follows immediately from part (b).
\end{proof}

The following remark is stated as a lemma since it will be used several
times in the paper.
\begin{lemma}\label{lemma: factors}
Let $X\stackrel{f}{\gfl}Y$ be a morphism in $\cat$ with $X\in\cbart$ and
assume that $f$ factors through $\tbar^\perp$ in $\cat$.
Then $f$ factors through $\tbar^\perp \cap \cbart$.
\end{lemma}

\begin{proof}
Let $\tbar_0 \stackrel{u}{\gfl} X$ be a minimal right
$\tbar$-approximation of $X$ in $\cat$. Complete
the morphism $u$ to a triangle $\tbar_0 \stackrel{u}{\fl}
X \stackrel{v}{\fl} Z \fl \shift \tbar_0$
in $\cat$. As shown in~\cite{BMRRT} (apply the functor $\cat(\tbar,-)$
to the triangle above) the cone $Z$ belongs to $\tbar^\perp$.
Moreover, the composition $fu$ vanishes since $f$ factors through
$\tbar^\perp$ and it follows that the morphism $f$ factors through $v$.
It remains to be checked that the object $Z$ lies in $\cbart$.
The triangle above shows that $Z\in\cbart\ast\add\shift\tbar$,
and we have:
\begin{eqnarray*}
 \cbart\ast \add\shift\tbar & = &
(\add T\ast \add\shift\tbar)\ast \add\shift \tbar \\
 & = & \add T \ast ( \add\shift\tbar\ast \add\shift \tbar) \\
 & = & \add T \ast \add\shift\tbar,
\end{eqnarray*}
where the last equality holds since $\shift \tbar$ is rigid.
\end{proof}

The following lemma, which is used in the proof
of Proposition~\ref{proposition: right adjoint},
is a particular case of \cite[IV.1 Theorem 2 (ii)]{MacLane}.
\begin{lemma}\label{lemma: MacLane}
Let $\bc$ be a category, and let $\ac$ be a full subcategory of $\bc$.
Suppose that, for any $B\in\bc$, there is an object $G_0B\in\ac$
and a morphism $G_0B\stackrel{\eta_B}{\gfl}B$ such that for all
$A\stackrel{f}{\gfl} B$ with $A\in\ac$, the morphsim $f$ lifts
uniquely through $\eta_B$. Then the inclusion $\ac\subseteq\bc$
has a right adjoint $G : \bc\fl\ac$ such that, for all $B\in\bc$,
$GB = G_0B$.
\end{lemma}
The functor $G$ of the previous lemma is defined on arrows as follows:
For any $B\stackrel{b}{\gfl}B'$ in $\bc$,
$Gb$ is the unique lift through $\eta_{B'}$ of the composition
$b\eta_B$:
\[\xymatrix{
G_0B \ar[r]^{\eta_B} \ar@{-->}[d]_{Gb} & B \ar[d]^b \\
G_0B' \ar[r]_{\eta_{B'}} & B'.
}\]

The following proposition is inspired
by~\cite{Beligiannis}:
\begin{prop}\label{proposition: right adjoint}
The inclusion of $\cbart$ into $\cat$ induces a fully faithful functor
$\frac{\cbart}{\tbar^\perp\cap\cbart} \stackrel{I}{\gfl}
\frac{\cat}{\tbar^\perp}$. Moreover, the functor $I$
admits an additive right adjoint $R$, such that, for all $X$ in $\cat$,
$RX=R_0X$, in the notation
of Lemma~\ref{lemma: cbart contravariantly finite}.
\end{prop}

\begin{proof}
The inclusion of $\cbart$ in $\cat$ induces a full functor:
$$
\xymatrix{
\cbart \drm \base_\qbar 				 &
\cat \base^Q					 \\
\cbart/_{\left(\tbar^\perp\cap\cbart\right)} \drp^I &
\cat/_{\left(\tbar^\perp\right)}.
}
$$
We first check that the functor $I$ is faithful.
This amounts to proving that if a morphism in $\cbart$
factors through $\tbar^\perp$ in $\cat$, then it already
factors through $\tbar^\perp$ in $\cbart$. This
follows from Lemma~\ref{lemma: factors}. In what follows,
we will identify $\cbart/(\tbar^\perp\cap\cbart)$ with the image of
$\cbart$ in $\cat/(\tbar^\perp)$.

Next, we prove the existence of a right adjoint. For this,
we use the particular case of~\cite[IV-1 Theorem 2 (ii)]{MacLane}, stated in Lemma~\ref{lemma: MacLane}.

Let $X\in\cat$. Consider the morphism
$R_0X\stackrel{\eta_X}{\gfl}X$ constructed in
Lemma~\ref{lemma: cbart contravariantly finite}.
We claim that $Q\eta_X$ is universal from $I$ to $X$,
in the sense of MacLane, i.e. any morphism in $\cat/(\tbar^\perp)$ from an object in $\cbart$ to $X$
factors uniquely through $Q\eta_X$ in $\cat/(\tbar^\perp)$.
Since $\eta_X$ is a right $\cbart$-approximation of $X$
in $\cat$, its image $Q\eta_X$ is a right
$\cbart/(\tbar^\perp\cap\cbart)$-approximation of $X$
in $\cat/(\tbar^\perp)$, so that we only have to prove uniqueness.

Let $Y\in\cbart$ and let $Y\stackrel{u}{\gfl}R_0X$ be a morphism
in $\cat$ such that $Q(\eta_X u) = 0$. Since the kernel of $Q$
is the ideal $(\tbar^\perp)$ of $\cat$, this means that
the composition $\eta_X u$ factors through $\tbar^\perp$.
Since its source belongs to $\cbart$, Lemma~\ref{lemma: factors}
shows that $\eta_X u$ factors through $\tbar^\perp\cap\cbart$.
Let $Y'\in\tbar^\perp\cap\cbart$ be such that the square:
$$
\xymatrix{
		  &
Y \bas_u \dr^a    &
Y' \bas^b         &
                  \\
Z \dr^\al         &
R_0X \dr^{\eta_X} &
X \dr		  &
\shift Z
}
$$
commutes.
Since $\eta_X$ is a right $\cbart$-approximation,
there exists a morphism $Y'\stackrel{c}{\gfl}R_0X$
with $b = \eta_X c$. We have $\eta_X(u - ca)=0$ so that
the morphism $u-ca$ factors through $\al$. By construction,
$Z\in\tbar^\perp$, therefore we have $u\in(\tbar^\perp)$,
which proves uniqueness.

Finally, we note that the functor $R$ is additive
since it is the right adjoint of the additive
functor $I$.
\end{proof}

If the category admits a Serre functor $S$,
then a dual version of
Proposition~\ref{proposition: right adjoint} will be of interest to us.
We first note that applying to $ST'$ the construction dual to that of $R_0$ gives, for any
$X\in\cat$, a triangle
$Z\stackrel{\al}{\gfl} X \stackrel{\eps_X}{\gfl} L_0 X \gfl \shift Z$,
where $L_0X$ belongs to $\add\susm S\tbar\ast\add ST'$,
$\eps_X$ is a minimal left
$\add\susm S\tbar\ast\add ST'$-approximation,
$\al$ factors through $^\perp(ST')=(T')^\perp$,
and $\shift Z$ belongs to $\tbar^\perp$. 

\begin{prop}\label{proposition: left adjoint}
 Assume that the category $\cat$ has a Serre functor $S$ and let
$\underline{\cat}(T')$ be the full subcategory $\add\susm S\tbar\ast\add ST'$
of $\cat$.
Then the inclusion of $\underline{\cat}(T')$ into $\cat$
induces a fully faithful functor
$\underline{\cat}(T')/(\tbar^\perp)
\stackrel{J}{\gfl} \cat/(\tbar^\perp)$. Moreover,
the functor $J$ admits an additive left adjoint $L$, such that $LX = L_0X$ for all $X\in\cat$.
\end{prop}

The only reason why we assume the existence of a Serre functor here is
that it converts a left perpendicular
subcategory into a right perpendicular subcategory. This allows us to view both
categories in Propositions \ref{proposition: right adjoint} and \ref{proposition: left adjoint}
as subcategories of the same category $\cat/(\tbar^\perp)$.

\subsection{Main result}\label{ssection: main result}
Our aim in this section is to prove that if $\cat$
has a Serre functor then the categories
$\cbart/(\shift T')$ and $\cbart/(T)$ are equivalent
(Theorem~\ref{theorem: equivalence of categories}).
This will then be used in the next section in order to compare the module categories over
the endomorphism algebras of $T$ and $T'$.

We need the following key lemma,
which will often be used throughout the paper.
\begin{lemma}\label{lemma: compute perps}
 We have:
\begin{enumerate}
 \item[(a)] $\cbart = T\ast\shift\tbar = \tbar\ast\shift T'$;
 \item[(b)] $\cbart\cap\tbar^\perp = \add\shift T'$;
 \item[(c)] if $\cat$ has a Serre functor $S$,
then $(\susm S\tbar \ast ST')\cap \tbar^\perp = \add\susm ST$.
\end{enumerate}
\end{lemma}

\begin{proof}
(a) The exchange triangle shows that $T\in\tbar\ast\shift T'$.
We thus have
\begin{eqnarray*}
 T\ast\shift\tbar & \subseteq & (\tbar\ast\shift T')\ast\shift\tbar \\
 & = & \tbar\ast(\shift T'\ast\shift\tbar) \\
 & = & \tbar\ast\shift T'.
\end{eqnarray*}
The reverse inclusion is obtained by applying this inclusion to $\shift T'$
(instead of $T$) in the opposite category.

(b) immediately follows from (a).

(c) also follows from (a):
\begin{eqnarray*}
 (\susm S\tbar \ast ST')\cap \tbar^\perp & = & (\susm ST\ast S\tbar)\cap\,^\perp S\tbar \text{ (by (a))}\\
 & = & \susm ST.
\end{eqnarray*}
\end{proof}

Assume that $\cat$ has a Serre functor $S$.
Recall that we write $\cbart$ (respectively, $\underline{\cat}(T')$)
for the full subcategory $T\ast\shift\tbar$
(respectively, $\susm S\tbar \ast ST'$) of $\cat$.
By Proposition~\ref{proposition: right adjoint},
Proposition~\ref{proposition: left adjoint} and Lemma~\ref{lemma: compute perps},
we have a pair of adjoint functors $(G,H)$,
where $G=JI$ and $H=RJ$. Since $I,J,L$ and
$R$ are additive, so are $G$ and $H$.
$$
\xymatrix{
                      &
\cat/(\tbar^\perp)
 \ar@<.5ex>[dl]^R
 \ar@<-.5ex>[dr]_L    &
                      \\
\cbart / (\shift T')
 \ar@<.5ex>[ur]^I 
 \ar@<.5ex>[rr]^G     &
                      &
\underline{\cat}(T') / (\susm ST)
 \ar@<-.5ex>[ul]_J
 \ar@<.5ex>[ll]^H
}
$$

\begin{rk}
We write $\tau$ for the Auslander--Reiten translation
$\tau= S\shift^{-1}$ (see \cite[\S I.2]{rvdb}).
Then, by Lemma~\ref{lemma: compute perps} we have that $\underline{\cat}(T') = \tau \cbart$.
\end{rk}

\begin{theo}\label{theorem: equivalence of categories}
Assume that $\cat$ has a Serre functor $S$.
Then the functors $G$ and $H$ are quasi-inverse equivalences of categories.
In particular, the categories $\cbart/(\shift T')$ and $\cbart/(T)$ are equivalent.
\end{theo}

\begin{proof}
The construction would be simplified
if we had that, if $X$ belongs to $T\ast\shift \tbar$,
then the left $\add \shift^{-1}S\tbar \ast \add ST'$-approximation $X \stackrel{\eps_X}{\gfl} L_0X$ of $X$ (see Proposition~\ref{proposition: left adjoint}
and the paragraph before it) is also a minimal right
$T\ast\shift \tbar$-approximation of $L_0X$.
However, this cannot be expected to hold in general (take $X$ to be $\shift T'$, for instance).

We can modify this approach in the following way.
First, since the functors $G$ and $H$ are additive,
we may assume that $X$ is indecomposable. This will help in proving that $X$ is a summand of $R_0L_0X$.
Second, we will add to $X$ a minimal right $\add(\shift T')$ approximation $\shift T'_0$ of $L_0X$.
This will be needed in order to get a right approximation of $L_0 X$, while being harmless since the objects
$X$ and $X\oplus\shift T'_0$ are isomorphic in $\catt$.

So, take an indecomposable object $X\in T \ast \shift \tbar$ and assume that $X$ does not belong
to $\add\shift T'$ (otherwise, $X$ would be isomorphic to $0$ in $\catt$. Consider the triangle
$Z \stackrel{\al}{\gfl} X \stackrel{\eps_X}{\gfl} L_0X \stackrel{\beta}{\gfl} \shift Z$
in $\cat$, constructed in the paragraph before Proposition~\ref{proposition: left adjoint}, where
$\al\in \left((T')^\perp\right)$
and $\shift Z \in \tbar^\perp$. Let $\shift T'_0 \stackrel{p}{\gfl} L_0X$ be a minimal
right $\add\shift T'$-approximation of $L_0X$ in $\cat$. We claim that
$X\oplus\shift T'_0 \stackrel{[\eps_X p]}{\gfl} L_0X$ is a right
$T\ast\shift\tbar$-approximation of $L_0X$ in $\cat$.
Let $X'\stackrel{f}{\gfl}L_0X$ be a morphism in $\cat$, with
$X'\in T\ast\shift\tbar$. By assumption, there is a triangle
$T'_1 \stackrel{a}{\gfl} \tbar_0 \stackrel{b}{\gfl} X' \stackrel{c}{\gfl} \shift T'_1$
in $\cat$, with $T'_1\in\add T'$ and $\tbar_0\in\add\tbar$.
Since $\tbar_0$ is in $\add\tbar$ and $\shift Z$ is in $\tbar^\perp$,
the composition $\beta f b$ vanishes and $f$ induces a morphism
of triangles:
$$
\xymatrix{
\tbar_0 \dr^b \bas_{\susm h}    &
X' \dr^c \bas_f \bgex_v         &
\shift T'_1 \dr^{-\shift a}
\bas_g \bdro \bgex_u            &
\shift \tbar_0 \bas_h           \\
X \dr_{\eps_X}                  &
L_0X \dr_\beta                  &
\shift Z \dr_{-\shift\al}       &
\shift X
}
$$
Since $\al$ factors through $(T')^\perp$, we have
$(-\shift\al) g = 0$ and there exists
$\shift T'_1 \stackrel{u}{\gfl}L_0X$ such that
$g = \beta u$. This implies $\beta(f-uc)=0$ so that
there exists $X'\stackrel{v}{\gfl}X$ such that
$f = uc+\eps_X v$. The composition $uc$ is in the ideal
$(\shift T')$ and thus factors through $p$, i.e.\
there exists $w$
making the following square commute
$$
\xymatrix{
X'\basp_w\dr^c      &
\shift T'_1 \bas_u  \\
\shift T'_0 \dr^p   &
L_0X
}
$$
We thus have $f = [\eps_X \; p] \left[^v_w\right]$ and therefore
the morphism $[\eps_X \; p]$ is a right $T\ast\shift\tbar$-approximation
of $L_0X$ in $\cat$.

Since also $R_0L_0X\stackrel{\eta_{L_0X}}{\gfl}L_0X$ is a right $T\ast\shift\tbar$-approximation
of $L_0X$ in $\cat$, we can write $R_0L_0X$ as a direct sum $X'\oplus X''$, and
$\eta_{L_0X} = [\eta'\; 0] : X'\oplus X'' \fl L_0X$, where $X'\stackrel{\eta'}{\gfl}L_0X$
is a minimal right $T\ast\shift\tbar$-approximation. Moreover, we have $X''\in\add\shift T'$, since
in the triangle $Z'\fl X'\oplus X''\fl L_0X \fl$ given by Lemma~\ref{lemma: cbart contravariantly finite},
$Z'$ belongs to $\tbar^\perp$, and $X''$ is a summand of $Z'$. Thus
$X''$ belongs to $\in\tbar^\perp\cap\cbart$ which is $\add\shift T'$ by Lemma~\ref{lemma: compute perps}.

Now $X'$ is a summand of the approximation $X\oplus \shift T'_0$. Moreover, $X'$ contains $X$
as a summand. Otherwise, we would have $R_0L_0X \in \add\shift T'$,
which implies $L_0X\in\tbar^\perp\cap\underline{\cat}(T')=\add\susm ST'$
(by applying the functor $\cat(\tbar,-)$ to the triangle $Z'\fl R_0L_0X\fl L_0X\fl$),
which dually implies $X\in\add\shift T'$ (note that we assumed $X\notin\add\shift T'$).

As a consequence, given a lift $\xymatrix{& X \ar[d]^{\eps_X} \ar@{-->}[dl]_{\widetilde{\vph}_X}
\\ R_0L_0X \ar[r]^{\eta_{L_0X}} & L_0X}$, the image $\vph_X$ of $\widetilde{\vph}_X$
in $\cbart/(\shift T')$, which is independent of the choice of $\widetilde{\vph}_X$
by Proposition~\ref{proposition: right adjoint}, is an isomorphism (and $\eps_X$
is a minimal right $\cbart$-approximation of $L_0X$ in $\cat/(\shift T')$).

Let us check that we have defined a natural isomorphism
$\vph : 1 \fl HG$.

Let $X\stackrel{f}{\gfl}Y$ be a morphism in $\cbart$.
By construction, there is a diagram in $\cat$
$$\xymatrix@!@-3pc{
X \ar[rr]^{\widetilde{\vph}_X} \ar[ddd]_f \ar[dr]
&& R_0L_0X \ar[dl] \ar[ddd]^{HGf} \\
& L_0X \ar[d]^{Gf} & \\
& L_0Y & \\
Y \ar[ur] \ar[rr]_{\widetilde{\vph}_Y}
&& R_0L_0Y \ar[ul]_\eta
}$$
were we write $\eta$ for $\eta_{L_0Y}$
and where the inner two triangles and the inner
two squares commute.
We thus have $\eta (HGf \circ \widetilde{\vph}_X - \widetilde{\vph}_Y\circ f) = 0$,
and $HGf \circ \widetilde{\vph}_X - \widetilde{\vph}_Y\circ f$ factors through
$\overline{Z} \fl R_0L_0Y$ in the triangle
$\overline{Z}\fl R_0L_0Y \stackrel{\eta}{\fl} L_0Y \fl$,
where $\overline{Z}\in\tbar^\perp$.
This shows that $HGf \circ \widetilde{\vph}_X - \widetilde{\vph}_Y\circ f$ 
factors through $\tbar^\perp$. Since $X$ lies in $\cbart$, we can apply
Lemma~\ref{lemma: factors}. The morphism $HGf \circ \widetilde{\vph}_X - \widetilde{\vph}_Y\circ f$
factors through $\tbar^\perp\cap\cbart$, which is $\add\shift T'$ by Lemma~\ref{lemma: compute perps}.
As a consequence, $\vph$ is a natural transformation.

By duality, there is a natural isomorphism $GH \fl 1$; and the functors $G$ and $H$ are quasi-inverse equivalences of categories.
\end{proof}

\subsection{A module-theoretic interpretation}
In this section we assume that $\cat$ 
has a Serre functor $S$. 
In this case, the assumptions of functorial finiteness
(see Section~\ref{s:setup}) are automatically satisfied
for all rigid objects (but have to be added in the case of rigid subcategories).
We write $D$ for the duality functor $\homph_k(-,k)$.
Recall that $T\in\cat$ is a basic rigid object, and
$R$ is a direct summand of $T$, with $T=\tbar\oplus R$.
We write $\tbar=T_1\oplus\cdots\oplus T_n$ and
$R=T_{n+1}\oplus\cdots\oplus T_m$, where the $T_i$
are indecomposable. Recall also that $\shift R^*$ is the
cone of a minimal right $\add\tbar$-approximation
of $R$.
We have $T'=\tbar\oplus R^\ast =
T'_1\oplus\cdots\oplus T'_m$
where $T'_i = T_i$ if $i\leq n$.
Define $\Lambda$ (respectively, $\Lambda'$), to be the endomorphism algebra $\End_\cat(T)^\text{op}$,
(respectively, $\End_\cat(T')^\text{op}$).

Let $S_j$ be the simple top of the indecomposable projective $\Lambda$-module
$\cat(T,T_j)$, and let $S'_j$ be the simple socle of the indecomposable injective $\Lambda'$-module
$D\cat(\shift T'_j, \shift T')$.
We consider the exact categories $\ec$ and $\ec'$ defined as follows.
The category $\ec$ (respectively, $\ec'$) is the full
subcategory of
$\operatorname{mod}\Lambda$, (respectively, $\operatorname{mod}\Lambda'$)
whose objects $M$ (respectively, $N$), satisfy $\ext^1_\Lambda(M,S_j) = 0$,
(respectively, $\ext^1_{\Lambda'}(S'_j,N)=0$)
for all $j > n$.

\begin{rk}\label{remark: exchange triangles}
For each indecomposable summand $R_i$ of $R$,
let $R_i^\ast \fl \ubar_i \fl R_i \fl \shift R_i^\ast$
be a triangle in $\cat$, with $\ubar_i\fl R_i$ a minimal right $\tbar$-approximation.
Then (as in~\cite{BMRRT}) the object $R_i^\ast$ is indecomposable. Moreover,
$\oplus_i \ubar_i \fl \oplus_i R_i$ is a minimal right $\tbar$-approximation of $R$.
As a consequence, $R^\ast$ is isomorphic to $\oplus_i R_i^\ast$. This shows that
the basic objects
$R$ and $R^\ast$ have the same number of indecomposable summands.
\end{rk}

We can now restate Theorem~\ref{theorem: equivalence of categories}
in module-theoretic terms:

\begin{theo}\label{theorem: equivalence module version}
Suppose that $\cat$ has a Serre functor. Then
there is an equivalence of categories:
$$\ec / \add\cat(T,\shift R^\ast) \simeq \ec' / \add D\cat(R,\shift T').$$
\end{theo}

The proof will be given later in this section.
We note that, if $\cat$ is 2-Calabi--Yau, then the modules $D\cat(R,\shift T')$
and $\cat(T',\shift R)$ are isomorphic.
We also note that although the statement of the
equivalence does not need a Serre functor,
the existence is needed in the proof, in order to
apply Theorem~\ref{theorem: equivalence of categories}.

In order to prove
Theorem~\ref{theorem: equivalence module version},
we will need the following two lemmas.

\begin{lemma}\label{lemma: ec}
The functor $\cat(T,-)$ induces a fully faithful functor
\[\cbart / (\shift\tbar) \gfl\operatorname{mod}\Lambda.\]
Its essential image is $\ec$.
\end{lemma}

\begin{proof}
Let $X\stackrel{f}{\gfl}Y$ be a morphism in $\cat$ factoring through $\add\shift T$.
Recall that $\cbart=T\ast \shift \tbar$.
Assume that $X$ belongs to $\cbart$, and let  $\vbar_1 \fl V_0\fl X \fl \shift \vbar_1$ be a triangle in $\cat$
with $V_0\in\add T$ and $\vbar_1\in\add\tbar$.
Since $T$ is rigid and $f$ factors through $\add\shift T$, the composition $V_0\fl X\fl Y$ vanishes,
and $f$ factors through $\shift\vbar_1$. This implies the first part of the lemma (the fullness of $\cat(T,-)$ follows
from~\cite[Prop. 6.2]{IY}; see also~\cite[Lemma 4.3]{BMloc1}).

For any $M$ in $\operatorname{mod}\Lambda$, let $X\in\ct$ be such that
$X$ has no summands in $\add\shift T$ and $\cat(T,X)\simeq M$.
Let $U_\beta \fl U_\al \fl X \fl \shift U_\beta$ be a triangle with
$U_\al, U_\beta \in \add T$ and $U_\al \fl X$ right-minimal. Then
$\cat(T,U_\beta)\fl\cat(T,U_\al)\fl\cat(T,X)\fl0$ is a minimal
projective presentation of $\cat(T,X)$, and
the dimension of $\ext^1_\Lambda(M,S_j)$ is the multiplicity
of $P_j$ in $\cat(T,U_\beta)$.
\end{proof}

Dually, we obtain the following:

\begin{lemma}\label{lemma: ecprime}
The functor $D\cat(-,\shift T')$ induces a fully faithful functor
\[\cbart / (\tbar) \gfl\operatorname{mod}\Lambda'.\]
Its essential image is $\ec'$.
\end{lemma}

\begin{proof}
The proof is dual to that of Lemma~\ref{lemma: ec}.
We use the description $\cbart=\tbar\ast \shift T'$
from Lemma~\ref{lemma: compute perps}, and note that any triangle
$U'_1 \fl \ubar_0 \fl X \fl \shift U'_1$, where $U'_1$ belongs to $\add T'$
and $\ubar_0$ belongs to $\add \tbar$, gives rise to an injective co-presentation
$0 \fl D\cat(X,\shift T') \fl D\cat(\shift U'_1,\shift T') \fl D\cat(\shift \ubar_0,\shift T')$.
\end{proof}

\noindent \emph{Proof of Theorem~\ref{theorem: equivalence module version}}:
By Lemma~\ref{lemma: ec}, the functor $\cat(T,-)$ induces
an equivalence of categories from $\cbart / (\shift\tbar)$
to $\ec$. Since $\cbart/ (\shift T')$ is isomorphic to
$\big(\cbart/(\shift\tbar) \big) / (\shift R^\ast)$,
the functor $\cat(T,-)$ induces an equivalence of categories
from $\cbart/ \add \shift T'$ to $\ec / \add\cat(T,\shift R^\ast)$.
Dually, one can use Lemma~\ref{lemma: ecprime} to notice that
the functor $D\cat(-,\shift T')$ induces an equivalence of categories
from $(\tbar\ast T')/\add T$ to $\ec'/\add D\cat(R,\shift T')$.
The statement now follows from Theorem~\ref{theorem: equivalence of categories}.\qed
\vskip 0.2cm
There are two particular cases of Theorem~\ref{theorem: equivalence module version} that 
are worth noting. They are weak forms of nearly-Morita equivalences that we call pseudo-Morita equivalences.
They occur in the case where $R$ is indecomposable,
i.e. $m=n+1$, and we make this assumption for the rest
of the section. Note that $R=T_m$ and $R^{\ast}=T'_m$.

Let $Q_m$ be the $\Lambda$-module
$\cat(T,\shift R^{\ast})=\cat(T,\shift T'_m)$
appearing in Theorem~\ref{theorem: equivalence module version}. Similarly, we have the $\Lambda'$-modules
$Q'_m=D\cat(R,\shift T')=D\cat(T_m,\shift T')$.
Then we have the following:

\begin{lemma} \label{lemma: QSm}
Suppose that $R$ is indecomposable.
Let $e$ is the idempotent for $\Lambda$ corresponding
to $T_m$. Then we have the isomorphism
$$Q_m\simeq \Lambda / \Lambda(1-e)\Lambda.$$
Furthermore, $Q_m$ is a simple object of $\ec$.
Dually, let $e'$ be the idempotent for $\Lambda'$ corresponding to $\shift T'_m$. Then we have the
isomorphism
$$Q'_m\simeq \Lambda' / \Lambda(1-e')\Lambda'.$$ Furthermore, $Q'_m$ is an indecomposable
$\Lambda'$-module and a simple object of $\ec'$.
\end{lemma}

\begin{proof}
The long exact sequences associated with the exchange triangle
$T'_m \fl \ubar_m \fl T_m \fl \shift T'_m$ of Remark~\ref{remark: exchange triangles}
show that the functor $\cat(-,\shift T'_m)$
vanishes on $\add (\tbar)$ and that the $\Lambda$-module
$\cat(T_m,\shift T'_m)$ is isomorphic to the module
$\cat/(\add\tbar)(T_m,T_m)$.
We thus have an isomorphism of $\Lambda$-modules:
$$
\frac{\Lambda}{\Lambda(1-e)\Lambda}\simeq
\frac{\cat(T,T)}{(\tbar)}\simeq
\frac{\cat}{(\add\tbar)}(T_m,T_m) \simeq
\cat(T_m,\shift T'_m)\simeq Q_m.
$$
Similarly, using the exchange triangle
as above, we obtain an isomorphism between the
$\Lambda'$-modules
$D\cat(T_m,\shift T')$ and $D\cat/(\add\shift\tbar)(\shift T',\shift T')$,
the latter being isomorphic to $Q'_m$.

Since $Q_m$ is projective over $\Lambda/\Lambda(1-e)\Lambda$, we have
$$\ext^1_{\Lambda}(Q_m,S_m)\simeq \ext^1_{\Lambda/\Lambda(1-e)\Lambda}(Q_m,S_m)=0,$$
and therefore $Q_m$ lies in $\ec$.
If $N$ is a non-trivial submodule of $Q_m$
lying in $\ec$, then $N$ is also a $\Lambda/\Lambda(1-e)\Lambda$-module satisfying $\ext^1_{\Lambda(1-e)\Lambda}(N,S_m)=0$. Since $S_m$ is the only simple
$\Lambda(1-e)\Lambda$-module, $N$ is projective over
$\Lambda(1-e)\Lambda$, so it must equal $Q_m$. It follows that $Q_m$ is a simple object of the
exact category $\ec$.
Since $\ec$ is closed under direct
summands, it now follows that
$Q_m$ is an indecomposable $\Lambda$-module.
The proofs of the duals of these last two statements are similar.
\end{proof}

\begin{cor}\label{corollary: loop}
Suppose that $\cat$ satisfies the assumptions in
Section~\ref{s:setup} and that it has a Serre functor.
Suppose further that
$R$ is indecomposable,
Then there is an equivalence of categories:
$$\ec / \add Q_m \simeq \ec' / \add Q'_m.$$
\end{cor}
\begin{proof}
This follows from Theorem~\ref{theorem: equivalence module version} and Lemma~\ref{lemma: QSm}.
\end{proof}

\begin{cor}\label{corollary: no loop}
Suppose that the assumptions in
Corollary~\ref{corollary: loop} hold,
and, in addition, that the Gabriel quiver of
$\Lambda$ has no loop at the vertex corresponding to $R$.
Then there is an equivalence of categories:
$$\ec / \add S_m \simeq \ec' / \add S'_m.$$
\end{cor}
\begin{proof}
This is a particular case of Corollary~\ref{corollary: loop}. Indeed,
by \cite[Proposition 2.6 (1)]{IY} $R^\ast$ is also indecomposable
and has no loop. This implies that $Q_m$ and $Q'_m$ are isomorphic to
$S_m$ and $S'_m$ respectively.
\end{proof}

\section{Localisation}
\subsection{Notation and statement of main results}
\label{ssection: localisation}
We continue with the assumptions and notation from Section~\ref{s:setup}. We do not assume here that
$\cat$ has a Serre functor, except in Corollary~\ref{corollary: modllocalisation}.
Also, contrary to \cite{BMloc1}, we do not make any
skeletal smallness assumption. This is because all the localisations
that we consider are shown to be equivalent to a 
subquotient of $\cat$. Therefore no set-theoretic 
difficulties arise, and the localisations we consider 
are all categories without passing to a higher universe.

Recall that, by~\cite{KR1,BMloc1}, the
functor $\cat(T,-)$ induces an equivalence of
categories from $\cat(T)/\shift T$ to $\modl$.
In particular, it is dense and full when restricted to $\cat(T)$.

\begin{defi}
Let $\bc$ be the full subcategory of $\modl$ given by 
the (essential) image of $\tbar^\perp$ under $\cat(T,-)$. Let $\sc_{\bc,0}$ be the class of all epimorphisms
$f\in \modl$ whose kernel belongs to $\bc$.
Dually, we let $\bc'$ be the full subcategory of $\modl'$ given by the (essential) image of
${}^{\perp}\shift \tbar$ under
$D\cat(-,\shift T')$
and set $\sc_{0,\bc'}$ to be the
class of all monomorphisms $g\in \modl'$
whose cokernel belongs to $\bc'$.
\end{defi}

Let $F$ be the composition of the fully faithful functor $\cbart/\shift\tbar \fl \cat(T)/\shift T \fl \modl$
and the localisation functor $\modl \stackrel{L_{\sc_{\bc,0}}}{\gfl}(\modl)_{\sc_{\bc,0}}$.
Then, since
$\cat(T,\shift R^\ast)$ belongs to $\bc$, we have that
$F(\shift R^\ast )\simeq 0$ in $(\modl)_{\sc_{\bc,0}}$.
Hence, $F$ induces a functor $\fbar$ as in the following diagram:
\begin{equation}
\label{e:Fdiagram}
\xymatrix{
\cbart/(\shift\tbar) \ar@{->>}[d] \ar@{^(->}[r] \ar[drr]^F
& \cat(T)/(\shift T) \ar[r]^\simeq
& \modl \ar[d]^{L_{\sc_{\bc,0}}} \\
\cbart/(\shift T') \ar@{-->}[rr]^{\fbar}
&& (\modl)_{\sc_{\bc,0}}.
}
\end{equation}

Our main aim in this section is to show that
the following holds:

\begin{theo}\label{theorem: fbar equivalence}
The functor $\fbar : \cbart/(\shift T') \gfl (\modl)_{\sc_{\bc,0}}$ is an equivalence
of categories. Dually, there is an equivalence
$\cbart/(T) \gfl (\modl')_{\sc_{0,\bc'}}$.
\end{theo}

This has two key corollaries, which we state below,
after a lemma needed in the proof of the first one.

\begin{lemma}\label{lemma: Tbarintersection}
For any object $M\in\bc$, there exists
$X\in\cat(T)\cap\tbar^\perp$ such that
$\cat(T,X)\simeq M$.
\end{lemma}
\begin{proof}
For any object $M\in\bc$, there exists an object $Y\in\tbar^\perp$ such that $\cat(T,Y) \simeq M$.
By~\cite[Lemma 3.3]{BMloc1}, there is a triangle
$Z\fl X\fl Y \stackrel{\eps}{\fl}\shift Z$ in $\cat$, where
$X\in\cat(T)$, $Z\in T^\perp$ and $\eps\in(T^\perp)$.
Then we have $\cat(T,X)\simeq\cat(T,Y)\simeq M$, which can be
seen by applying the functor $\cat(T,-)$ to the triangle above.
Moreover, $X$ belongs to $\tbar^\perp$, since both $Z$ and $Y$ belong to $\tbar^{\perp}$.
\end{proof}

\begin{cor} \label{corollary: gentheoremB}
There is an equivalence of categories
$$(\modl)_{\sc_{\bc,0}} \simeq \ec / \add\cat(T,\shift R^\ast)$$
and, dually, an equivalence of categories
$$(\modl')_{\sc_{0,\bc'}} \simeq \ec' / \add D\cat(R,\shift T').$$
\end{cor}

\begin{proof}
For the first statement, combine Theorem~\ref{theorem: fbar equivalence} with Lemma~\ref{lemma: ec}, and for
the second statement, combine Theorem~\ref{theorem: fbar equivalence} with Lemma~\ref{lemma: ecprime}.
\end{proof}

\begin{proof}[Proof of Theorem~\ref{theoremB}]
We set $\cat$ to be an acyclic cluster category and
$T$ a rigid object in $\cat$. 
We consider the case $m=n+1$ and $R=T_m$ is indecomposable. As in the proofs of Corollaries~\ref{corollary: loop} and~\ref{corollary: no loop}, $\cat(T,\Sigma R^*)\simeq Q_m\simeq S_m$ in this case.
In particular, there are no loops in the quiver of $\End_{\cat}(T)$ at the vertex corresponding to $S_m$. 

By Lemma~\ref{lemma: Tbarintersection},
we have $\bc=\cat(T,\tbar^{\perp})=\cat(T,\cat(T)\cap \tbar^{\perp})$.
For $j=1,\ldots ,n+1$, let $P_j=\cat(T,T_j)$ be the
$j$th indecomposable projective in $\modl$.
Then an object $X$ in $\cat(T)$ lies in $\tbar^{\perp}$
if and only if $\homph_{\Lambda}(P_j,\cat(T,X))=0$
for $j=1,2,\ldots ,n$, which holds if and only if
$\cat(T,X)$ lies in $\add(S_m)$. It follows that
$\bc=\add(S_m)$ in this case, and hence $\sc_{\bc,0}$
coincides with the class $\morone$ of morphisms considered
in the introduction. We see that the first statement
in Theorem~\ref{theoremB}
follows from the first statement in
Corollary~\ref{corollary: gentheoremB}. The second statement in Theorem~\ref{theoremB} follows from the
second statement in Corollary~\ref{corollary: gentheoremB}.
\end{proof}

\begin{cor} \label{corollary: modllocalisation}
If the category $\cat$ admits a Serre functor, then
there is an equivalence of categories:
$$(\modl)_{\sc_{\bc,0}}\simeq (\modl')_{\sc_{0,\bc'}}.$$
\end{cor}
\begin{proof}
Combine Theorem~\ref{theorem: fbar equivalence}
with Theorem~\ref{theorem: equivalence of categories}.
\end{proof}

\begin{proof}[Proof of Theorem~\ref{theoremA}]
Since an acyclic cluster category has a Serre functor,
Theorem~\ref{theoremA} follows from Corollary~\ref{corollary: modllocalisation} and the observations
in the proof of Theorem~\ref{theoremB} above.
\end{proof}

We shall also use Theorem~\ref{theorem: fbar equivalence} to show that the categories $\cat_{\sc}$ and $\cat_{\widetilde{\sc}}$ are isomorphic
(Theorem~\ref{theorem: more localisations}).
We also remark that Lemma~\ref{lemma: Z in CT}
may be of independent interest.

\subsection{Proof of Theorem~\ref{theorem: fbar equivalence}}
We show the first statement of the Theorem. The second statement follows from a dual argument.
In order to prove that $\overline{F}$ is full and dense, it is enough to prove that $F$
is full and dense. The functor $F$ is easily seen to be dense (Proposition~\ref{proposition: F full}).
Showing that it is full requires a bit more work (Lemmas~\ref{lemma: Z in CT} and~\ref{lemma: B iff XT}),
and in order to do so we describe, in Lemma~\ref{lemma: Stilde and SB0},
the category $(\modl)_{\sc_{B,0}}$ as a localisation of $\cat$.
We then show that the functor $\homph_\Lambda(U,-)$ induces a functor $(\modl)_{\sc_{\bc,0}}\gfl\modlbar$
(Lemma~\ref{lemma: modl loc to modlbar}).
Composing $F$ with this induced functor and
applying results from~\cite{BMloc1} then gives us the
faithfulness of $\overline{F}$ (Proposition~\ref{proposition: faithful}).

\begin{lemma}\label{lemma: B stable}
The full subcategory $\bc$ of $\modl$ is closed
under taking images and submodules.
\end{lemma}
\begin{proof}
Let $M\stackrel{u}{\fl}N$ be a morphism in $\modl$.
Then there are objects $X,Y\in \cat(T)$ such that
$\cat(T,X)\simeq M$ and $\cat(T,Y)\simeq N$, and a morphism $f:X\rightarrow Y$ such that
$\cat(T,f)\simeq u$.
We complete $f$ to a triangle
$$Z\stackrel{g}{\fl} X \stackrel{f}{\fl} Y \stackrel{h}{\fl} \shift Z.$$

If $M$ lies in $\bc$ and $u$ is an epimorphism then,
by Lemma~\ref{lemma: Tbarintersection}, we may take
$X$ in $\tbar^{\perp}$ and, by~\cite[Lemma 2.5]{BMloc1},
$h$ factors through $T^\perp$.
If $N$ lies in $\bc$ and $u$ is a monomorphism then,
by Lemma~\ref{lemma: Tbarintersection}, we may take
$Y$ in $\tbar^{\perp}$ and, by~\cite[Lemma 2.5]{BMloc1},
$f$ factors through $T^\perp$.

In either case, the result follows from applying the functor $\cat(\tbar,-)$ to this triangle.
\end{proof}

\begin{prop}\label{proposition: F dense}
The functor $F$ is dense.
\end{prop}

\begin{proof}
For any module $M\in\modl$, let $X$ be an object
in $\cat(T)$ such that $\cat(T,X)\simeq M$.
In Lemma~\ref{lemma: cbart contravariantly finite},
we constructed a triangle
$Z \gfl R_0X\stackrel{\eta_X}{\gfl}X\stackrel{g}{\gfl}\shift Z$,
with $R_0X\in\cbart$, $Z\in\tbar^\perp$, and $g\in(T^\perp)$.
We claim that the morphism
$\eta_X$ is inverted in $(\modl)_{\sc_{\bc,0}}$.
There is an exact sequence in $\modl$:
$$\cat(T,Z) \fl \cat(T,R_0X) \fl \cat(T,X) \stackrel{0}{\fl}\cat(T,\shift Z).$$
Therefore, the morphism $\cat(T,\eta_X)$ is surjective and $\cat(T,Z)$ surjects onto its kernel.
Since $\bc$ is closed under images (Lemma~\ref{lemma: B stable}), we may conclude that $\cat(T,\eta_X)$ belongs to $\sc_{\bc,0}$, and the claim is shown.
This shows that $M\simeq FR_0X$ in
$(\modl)_{\sc_{\bc,0}}$ and we are done.
\end{proof}

\begin{lemma}\label{lemma: Z in CT}
Let $Z\stackrel{u}{\gfl} X\stackrel{v}{\gfl} Y \stackrel{\eps}{\gfl} \shift Z$ be a triangle in $\cat$
with $X,Y\in\cat(T)$ and $\eps \in (T^\perp)$. Then $Z$ belongs to $\cat(T)$.
\end{lemma}

\begin{proof}
 Let $T^Z \stackrel{f}{\gfl}Z$ be a minimal right $\add T$-approximation.
Complete it to a triangle $U\fl T^Z \stackrel{f}{\fl} Z \stackrel{\delta}{\fl} \shift U$.
We note that, since $f$ is an approximation and $T$ is rigid, we have $\shift U\in T^\perp$,
as can be seen by applying the functor $\cat(T,-)$ to the triangle above, in a manner similar to
that of~\cite[Lemma 6.3]{BMRRT}
(a more general version of this assertion can be found in~\cite[Lemma 2.1]{Jorgensen-Wakamatsu}).

Since $Y$ belongs to $\cat(T)$, there is a triangle
$T_1^Y\fl T_0^Y \stackrel{a}{\fl}Y \stackrel{\eta}{\fl}\shift T_1^Y$.
By assumption, the composition $\eps a$ vanishes, so that there is
a morphism $T_0^Y\stackrel{b}{\fl}X$ such that $a = vb$.
We thus have a morphism of triangles
$$
\xymatrix{
T^Z \ar[r]^{\left[^0_1\right]} \ar[d]^f
& T^Z\oplus T_0^Y \ar[r]^{[0\,1]} \ar[d]^{[uf\,b]}
& T_0^Y \ar[r]^0 \ar[d]^a
& \shift T^Z \ar[d]^{\shift f} \\
Z \ar[r]^u
& X \ar[r]^v
& Y \ar[r]^\eps
& \shift Z
}
$$
that we complete to a nine diagram
$$
\xymatrix{
U \ar[d] \ar[r]
& V \ar[d] \ar[r]
& T_1^Y \ar[d] \ar[r]
& \shift U \ar[d] \\
T^Z \ar[r]^{\left[^0_1\right]} \ar[d]^f
& T^Z\oplus T_0^Y \ar[r]^{[0\,1]} \ar[d]^{[uf\,b]}
& T_0^Y \ar[r]^0 \ar[d]^a
& \shift T^Z \ar[d]^{\shift f} \\
Z \ar[r]^u \ar[d]^{\delta}
& X \ar[r]^v \ar[d]
& Y \ar[r]^\eps \ar[d]^{\eta}
& \shift Z \\
\shift U \ar[r]
& \shift V \ar[r]
& \shift T_1^Y
&
}
$$
Since $\shift U \in T^\perp$, the morphism
$T_1^Y\fl\shift U$ vanishes and the top triangle
splits. Thus $U$ is a summand of $V$ and it is enough to prove
that $V$ belongs to $\add T$. Since $X\in\cat(T)$, this amounts
to proving that the morphism $T^Z\oplus T_0^Y \fl X$
is an $\add T$-approximation. Let us thus prove the latter statement.
Let $W \stackrel{g}{\gfl}X$ be a morphism in $\cat$ with
$W\in\add T$ (the morphisms are illustrated in
the diagram below). Then the composition
$\eta v g$ is zero so that there is a morphism $W \stackrel{c}{\gfl}T_0^Y$
with $vg = ac$. This implies $vg = vbc$, and there is a morphism
$W\stackrel{d}{\gfl}Z$ such that $g-bc = ud$. Now $f$
is an $\add T$-approximation so that there is a morphism $W\stackrel{e}{\gfl}T^Z$
satisfying $d = fe$. The last two equalities give
$g = ufe + bc = [uf\;b]\left[^e_c\right]$
and we have shown that $[uf\;b]$ is an $\add T$-approximation.
\end{proof}

The following diagram shows the morphisms
$g,b,c,d$ and $e$.
$$
\xymatrix{
&&& W \ar@/^2pc/[ddll]^g
\ar@/_/@{-->}[dl]_c \ar@/_1.5pc/@{-->}[ddlll]_d
\ar@/_2pc/@{..>}[dlll]_e \\
T^Z \ar[r]^{\left[^0_1\right]} \ar[d]_f
& T^Z\oplus T_0^Y \ar[r]^{[0\,1]} \ar[d]_{[uf\,b]}
& T_0^Y \ar[r]^0 \ar@{-->}[dl]_b \ar[d]_a
& \shift T^Z \ar[d]^{\shift f} \\
Z \ar[r]_u
& X \ar[r]_v
& Y \ar[r]_\eps \ar[d]^\eta
& \shift Z \\
&& \shift T_1^Y &
}
$$

\begin{lemma}\label{lemma: B iff XT}
 Let $Z\stackrel{f}{\gfl}X$ be a morphism in $\cat$ with $Z\in\cat(T)$.
Then $\cat(T,f)$ factors through $\bc$ if and only if $f$ factors
through $\tbar^\perp$.
\end{lemma}

\begin{proof}
If $f$ belongs to the ideal $(\tbar^\perp)$, then $\cat(T,f)$
factors through $\bc$ by the definition of $\bc$. Let us prove the converse. Since $Z\in\cat(T)$,
there is a triangle $T_1\fl T_0\stackrel{g}{\fl} Z\fl \shift T_1$ in $\cat$
with $T_0,T_1\in\add T$. 
Assume that $\cat(T,f)$ belongs to $(\bc)$. Then there exists
$U\in\tbar^\perp$,
and there exist maps $\cat(T,Z)\stackrel{b}{\gfl}\cat(T,U)$ and $\cat(T,U)\stackrel{a}{\gfl}\cat(T,X)$ such that
$\cat(T,f) = a\circ b$.

We would like to lift $a$ and $b$ to morphisms in the category $\cat$. This cannot be done
in general, since the functor $\cat(T,-)$ is not full. Fortunately, it is full when restricted to $\cat(T)$.
We thus use~\cite[Lemma 3.3]{BMloc1} in order to replace the object $U$ by an object $U'$
whose image under $\cat(T,-)$ is isomorphic to that of $U$, but with the additional property that
$U'$ is in $\cat(T)$. Let us therefore apply~\cite[Lemma 3.3]{BMloc1} so as to get triangles
$Y_U \fl U'\fl U \stackrel{\eps}{\fl} \shift Y_U$ and $Y_X \fl X'\stackrel{u}{\fl} X \stackrel{\eta}{\fl} \shift Y_X$ in $\cat$,
where $U',X'$ belong to $\cat(T)$,
where $Y_U,Y_X$ belong to $T^\perp$, and where the morphisms
$\eps$ and $\eta$ factor through $T^\perp$. Since $U$ is in $\tbar^\perp$
and $Y_U$ in $T^\perp$, $U'$ is in $\tbar^\perp$ as well.

The modules $\cat(T,U)$ and $\cat(T,U')$ are isomorphic and
$\cat(T,u)$ is an isomorphism so that
there are morphisms $\cat(T,Z)\stackrel{b'}{\gfl}\cat(T,U')$
and $\cat(T,U')\stackrel{a'}{\gfl}\cat(T,X')$ satisfying
$\cat(T,u)\circ a'\circ b' = \cat(T,f)$. Now, the objects
$Z,U'$ and $X'$ all belong to $\cat(T)$ so that there exist
morphisms $\al,\beta$ in $\cat$ with $\cat(T,\al)=a'$
and $\cat(T,\beta)=b'$. We thus have the following diagram in $\cat$:
$$
\xymatrix{
T_0 \bas^g & U' \dr^\al & X' \bas_u \\
Z \ddr^f \hdr^\beta \bas && X \\
\shift T_1 \ar@<-0.6ex>@/^.7pc/@{-->}[urr] &&
}
$$
where the square $f-u\al\beta$ commutes up to a summand in $T^\perp$.
Since $T_0\in\add T$, the composition $(f-u\al\beta)g$ vanishes
and $f-u\al\beta$ factors through $\shift T_1$. This shows that
$f$ factors through $U'\oplus\shift T_1$ which belongs to $\tbar^\perp$,
and we are done.
\end{proof}

\begin{defi}
Let $\widetilde{\sc}$ be the class of morphisms $X\stackrel{s}{\gfl}Y$
in $\cat$ such that, for any triangle
$Z\stackrel{f}{\gfl}X\stackrel{s}{\gfl}Y\stackrel{g}{\gfl}\shift Z$
we have $f\in(\tbar^\perp)$ and $g\in(T^\perp)$.
Note that this is a weaker property than that defining
$\sc$ (where instead of the property
$f\in (\tbar^{\perp})$ we had $Z\in \tbar^{\perp})$.
Therefore $\sc\subseteq \widetilde{\sc}$.

Let $\cat\stackrel{L_{\widetilde{\sc}}}{\gfl}\cat_{\widetilde{\sc}}$
be the localisation functor with respect to the class $\widetilde{\sc}$.
\end{defi}

\begin{lemma}\label{lemma: Stilde and SB0}
 There is a commutative diagram
$$\xymatrix@!C{
\cat \ar[r]^{\cat(T,-)} \ar[d]_{L_{\widetilde{\sc}}}
& \modl \ar[d]^{L_{\sc_{\bc,0}}} \\
\cat_{\widetilde{\sc}} \ar[r]^{G'}
& (\modl)_{\sc_{\bc,0}},
}$$
where $G'$ is an equivalence of categories.
\end{lemma}

\begin{proof}
 It is proved in~\cite{BMloc1} that the functor $\cat(T,-) : \cat \fl \modl$
is a localisation functor for the class $\sc_T$ of morphisms $X\stackrel{f}{\gfl}Y$
such that, when completed to a triangle $Z\stackrel{g}{\gfl}X\stackrel{f}{\gfl}Y\stackrel{h}{\gfl}\shift Z$,
we have $g,h\in(T^\perp)$. Since this class is contained in the class $\widetilde{\sc}$,
it is enough to prove $\cat(T,\widetilde{\sc}) = \sc_{\bc,0}$.
Let $s$ be in $\widetilde{\sc}$. There is a triangle
$Z\stackrel{g}{\gfl}X\stackrel{s}{\gfl}Y\stackrel{h}{\gfl}\shift Z$ in $\cat$
with $g\in(\tbar^\perp)$ and $h\in(T^\perp)$. Applying the functor $\cat(T,-)$ gives
an exact sequence in $\modl$:
$$\cat(T,Z)\gfl \cat(T,X) \gfl \cat(T,Y) \stackrel{0}{\gfl} \cat(T,\shift Z),$$
where $\cat(T,g): \cat(T,Z)\fl \cat(T,X)$ factors through some $B\in\bc$.
Thus $\cat(T,s)$ is an epimorphism
and its kernel is isomorphic to a quotient of a submodule of $B$
(see Remark~\ref{remark: quotient of a submodule} below).
By Lemma~\ref{lemma: B stable} the subcategory $\bc$ is stable under taking images and submodules, so that $\cat(T,s)$ belongs to $\sc_{\bc,0}$.

Conversely, let $0\fl B\gfl M\stackrel{f}{\gfl}N\fl 0$ be a short exact sequence in $\modl$,
with $B\in\bc$. There is a morphism $X\stackrel{s}{\gfl}Y$ in $\cat$,
with $X,Y\in\cat(T)$ such that $\cat(T,s)\simeq f$. Complete it to a triangle
$Z\stackrel{u}{\gfl}X\stackrel{s}{\gfl}Y\stackrel{v}{\gfl}\shift Z$ in $\cat$.
Then $v\in(T^\perp)$ since $f$ is an epimorphism, and $\cat(T,u)$ factors through $B$ since $su = 0$.
Lemma~\ref{lemma: Z in CT} shows that $Z$ lies in $\cat(T)$ and we can apply Lemma~\ref{lemma: B iff XT}
to conclude that $u$ factors through $\tbar^\perp$.
\end{proof}

\begin{rk}\label{remark: quotient of a submodule}
Let $L\stackrel{g}{\gfl} M\stackrel{f}{\gfl} N$ be exact in an abelian category.
Assume that the morphism $g$ factors as $u\circ v$ through some object $B$.
Then the kernel of $f$ is isomorphic to a quotient of a subobject of $B$. 
\end{rk}

\begin{proof}
Let $K\stackrel{i}{\fl}M$ be a kernel for $f$. Since the sequence is exact,
there is an epimorphism $L\stackrel{g'}{\gfl}K$ such that $ig' = g$.
The morphism $v$ factors as in the following diagram:
$$
\xymatrix{
B'\; \dri^j & B \bdr^u & & & \\
L \bdre_{g'} \ar@{->>}[u]^p \hdr^v \ddr^g & & M \ddr^f & & N \\
 & K^{\phantom{n}} \hdri_i & & &
}
$$
The composition $fujp = fg$ vanishes so that $fuj = 0$ and there is
some $B'\stackrel{q}{\gfl}K$ so that $iq = uj$. It remains to show that
$q$ is an epimorphism. Since $i$ is a monomorphism, the equalities
$iqp=ujp=g=ig'$ imply $qp=g'$. Since the morphism
$g'$ is an epimorphism, $q$ is an epimorphism also.
\end{proof}

\begin{prop} \label{proposition: F full}
 The functor $F$ is full.
\end{prop}

\begin{proof}
Let $X\in\cbart$. Then there is a triangle
$\tbar_1 \stackrel{\beta}{\gfl} T_0\stackrel{\al}{\gfl} X\stackrel{\gamma}{\gfl} \shift \tbar_1$ in $\cat$.
Consider a hook diagram
$$
\xymatrix{
& \cat(T,X) \ar[d]^{\cat(T,f)} \\
\cat(T,U) \ar[r]_{\cat(T,s)} & \cat(T,V)
}
$$
in $\modl$, with
$U,V\in\cat(T)$ and $\cat(T,s)\in\sc_{\bc,0}$.
Let us prove that the morphism $\cat(T,f)$ lifts through
the morphism $\cat(T,s)$.
The proof of Lemma~\ref{lemma: Stilde and SB0} shows that
$s$ belongs to $\widetilde{\sc}$.
We thus have a triangle
$W \stackrel{g}{\gfl} U \stackrel{s}{\gfl} V \stackrel{h}{\gfl}\shift U$
in $\cat$ with $g\in(\tbar^\perp)$ and $h\in(T^\perp)$.
The composition $hf\al$ vanishes, so that $f$ induces a morphism of triangles
$$
\xymatrix{
& \tbar_1 \ar[r] \ar[d]^a \ar@/^1pc/@{-->}[dl]_{\susm b}
& T_0 \ar[r]^\al \ar[d]
& X \ar[d]_f \ar@{..>}[dl]_c \ar[r]^\gamma
& \shift\tbar_1\ar[d]^v \ar@{-->}[dl]_b \\
\susm V \ar[r]
& W \ar[r]^g
& U \ar[r]^s
& V \ar[r]^h
& \shift W.
}
$$
The morphism $g$ factors through $\tbar^\perp$ so that
the composition $ga$ is zero, giving the existence of a morphism $b$ such that $hb=v$. The equalities $hf = v\gamma = hb\gamma$
imply the existence of a morphism $c$ such that $f = b\gamma + sc$.
Therefore $\cat(T,s)\circ\cat(T,c) = \cat(T,f)$.
We can conclude by induction on the number of hooks
in a morphism from $\cat(T,X)$ to $\cat(T,Y)$.
\end{proof}

We write $U$ for $\cat(T,\tbar)$. Define
$\lbar$ to be the endomorphism algebra
of $U$ in $\modl$. Then $\lbar\simeq\End_{\cat}(\tbar)$.

\begin{lemma}\label{lemma: modlbar}
The diagram
$$\xymatrix@!C{
\cbart / (\shift\tbar) \ar[r]^{\cat(T,-)}
\ar[dr]_{\cat(\tbar,-)}
& \modl \ar[d]^{\homph_\Lambda(U,-)} \\
& \modlbar
}$$
commutes up to a natural isomorphism.
\end{lemma}

\begin{proof}
 For any $X\in\cbart$, define a map
$\vph_X : \cat(\tbar,X)\gfl \homph_\Lambda\big(\cat(T,\tbar),\cat(T,X)\big)$
by $\vph_X(\al) = \cat(T,\al)$.
Then $\vph_X$ is $\lbar$-linear, since $\cat(T,-)$ is a covariant functor,
and it is an isomorphism, since $\cat(T,-) : \cbart/(\shift\tbar) \fl \modl$ is
fully faithful.
The transformation $\vph$ is easily seen to be natural: Let
$X\stackrel{f}{\gfl}Y$ be a morphism in $\cat$. Write $f_\ast$
for the image of $f$ under the functor $\homph_\Lambda(U,-)\circ\cat(T,-)$.
We have to check that $\phi_Y\circ\cat(\tbar,f) = f_\ast\circ\phi_X$.
The left-hand side of this equality sends a morphism $\tbar\stackrel{\al}{\gfl}X$
to the map sending $u\in\cat(T,\tbar)$ to $(f\circ\al)\circ u$,
while the right-hand side sends $\al$ to
$u\mapsto f\circ(\al\circ u)$.
\end{proof}

\begin{lemma}\label{lemma: modl loc to modlbar}
The functor $\homph_\Lambda(U,-)$ induces a functor
$(\modl)_{\sc_{\bc,0}}\gfl\modlbar$.
\end{lemma}

\begin{proof}
Suppose that the morphism $f=\cat(T,s)$ lies
in $\sc_{\bc,0}$.
Then $s$ belongs to $\widetilde{\sc}$ by Lemma~\ref{lemma: Stilde and SB0}.
In particular, $s$ is part of a triangle $(r,s,t)$ with $r,t\in (\tbar^\perp)$,
so that $\cat(\tbar,s)$ is an isomorphism (as proved in~\cite[Lemma 2.5]{BMloc1}).
Hence, by Lemma~\ref{lemma: modlbar}, $\homph_\Lambda(U,f)$ is an isomorphism.
\end{proof}

\begin{prop} \label{proposition: faithful}
 The functor $\fbar$ is faithful.
\end{prop}

\begin{proof}
 Assume that $Fu=Fv$ for some
$u,v : X\fl Y$ in $\cbart$.
Then Lemmas~\ref{lemma: modlbar} and~\ref{lemma: modl loc to modlbar}
imply $\cat(\tbar,u) = \cat(\tbar,v)$ and
\cite[Lemma 2.3]{BMloc1} implies $u-v\in(\tbar^\perp)$.
Since $X$ belongs to $\cbart = \tbar\ast\shift T'$, there is a triangle
$\tbar_\al \stackrel{w}{\fl} X \fl \shift T'_\beta \fl$, with
$\tbar_\al\in\add\tbar$ and $T'_\beta\in\add T'$. The composition
$(u-v)w$ vanishes so that $u-v\in(\shift T')$ and the functor $\fbar$ is faithful.
\end{proof}

\begin{proof}[Proof of Theorem~\ref{theorem: fbar equivalence}]
By Proposition~\ref{proposition: F dense},
$F$ is dense, and by Proposition~\ref{proposition: F full}, $F$ is full. Hence $\fbar$ is also full and dense.
By Proposition~\ref{proposition: faithful},
$\fbar$ is faithful, and
Theorem~\ref{theorem: fbar equivalence} follows.
\end{proof}

\subsection{More localisations}
In this section, we prove, under the assumptions as in Section~\ref{s:setup},
that the localisations $\cat_\sc$ and
$\cat_{\widetilde{\sc}}$ are isomorphic.
We note that this result does not seem to follow easily from Lemma~\ref{lemma: Stilde and SB0}, as one would expect by analogy
with~\cite[Section 4]{BMloc1}.

\begin{lemma}\label{lemma: summands}
The full subcategory $\cbart$ of $\cat$ is stable
under taking direct summands.
\end{lemma}

\begin{proof}
Let $X,X'\in\cat$ be so that
$X\oplus X'$ belongs to $\cbart$. Let
$U_0 \fl X$, $V_0\fl Y$ be minimal right $\add T$-approximations.
Then $U_0\oplus V_0 \fl X\oplus X'$ is a minimal right $\add T$-approximation.
When completing it to a triangle $W \fl U_0\oplus V_0 \fl X\oplus X' \fl$,
we thus have $W\in\add\tbar$. The nine lemma gives a commutative diagram
whose rows and columns are triangles in $\cat$:
$$
\xymatrix@-.5pc{
\susm X \dr \bas & \susm(X\oplus X') \dr \bas & \susm X' \bas \dr^0 & X \bas \\
U_1 \dr \bas & W \dr \bas & V_1 \dr \bas & \shift U_1 \bas \\
U_0 \dr \bas & U_0\oplus V_0 \dr \bas & V_0\dr^0 \bas & \shift U_0 \bas\\
X \dr & X\oplus X' \dr & X' \dr^0 & \shift X \\
}
$$
We want to show that the triangle in the second row splits, which
would then imply that $U_1$ and $V_1$, being summands of $W$, belong to $\add\tbar$.
The composition $\susm X' \fl V_1 \fl \shift U_1$ is zero so that
the morphism $V_1 \fl \shift U_1$ factors through $V_1\fl V_0$.
But there are no non-zero morphisms from $V_0$ to $\shift U_1$ since
$V_0\in\add T$ and $\shift U_1$ is the cone of the right $\add T$-approximation
$U_0\fl X$.
\end{proof}

\begin{lemma}\label{lemma: inverse}
Let $X$ and $Y$ be objects in $\cbart$, and let
$X\stackrel{s}{\gfl}Y$ be a morphism in $\widetilde{\sc}$.
Then there exist $\ubar\in\add\tbar$, and morphisms
$\shift\ubar\stackrel{c}{\gfl}Y$, $Y\stackrel{a}{\gfl}\shift\ubar$,
and $Y\stackrel{d}{\gfl}X$ in $\cat$ such that:
\begin{enumerate}
 \item The morphism $X\oplus\shift\ubar \stackrel{[s\;c]}{\gfl}Y$
is in $\sc$;
 \item The image in $\cat_\sc$ of the morphism $Y\stackrel{\left[^d_a\right]}{\gfl}
X\oplus\shift\ubar$ is inverse to $[s\;c]$.
\end{enumerate}
In particular, all morphisms in $\cbart$ which belong to $\widetilde{\sc}$
are inverted by the localisation functor $L_\sc: \cat\fl\cat_\sc$.
\end{lemma}

\begin{proof}
Let $X\stackrel{s}{\fl}Y$ be a morphism in
$\widetilde{\sc}$, with $X$ and $Y$ in $\cbart$.
Complete it to a triangle
$X\stackrel{s}{\fl}Y \stackrel{v}{\fl} \shift Z \stackrel{u}{\fl}\shift X$.
We first show how to define the object $\ubar$ and the morphisms $a,c,d$.
By assumption, the morphisms $v$ and $u$ factor through
$T^\perp$ and $(\shift\tbar)^\perp$, respectively.
There is a triangle $\ubar \fl U \stackrel{\al}{\fl} Y \stackrel{a}{\fl} \shift \ubar$,
with $U\in\add T$ and $\ubar\in\add\tbar$.
Since $v$ is in $(T^\perp)$, the composition $v\al$ vanishes
and there is a morphism $b$ as in the diagram below, such that
$v = ba$.
$$
\xymatrix{
& U \bas_\al & & \\
X \dr^s & Y \bas_a \dr^v \ar@/^.5pc/@{-->}[l]^d & \shift Z \dr^u & \shift X \\
& \shift\ubar \ar@/_.5pc/@{-->}[ur]_b \ar@/_1pc/@{-->}[u]_c & &
}
$$
The composition $ub$ also vanishes since $u$ factors through $(\shift\tbar)^\perp$.
Therefore, there is a morphism $c$ such that $b=vc$. Moreover, there is a morphism
$d$ such that $1_Y = sd +ca$. Indeed, we have the following equalities:
$vca = ba = v$ so that $1_Y-ca$ factors through $s$.

Before showing that (1) and (2) are satisfied, we need a bit of preparation.
We may complete the above diagram to the following
commutative diagram, whose rows and columns are
triangles in $\cat$:
$$
\xymatrix{
X \ar[r]^s & Y \ar[r]^v & \shift Z \ar[r]^u & \shift X \\
Z' \ar[u]_e \ar[r]_f & \shift \ubar \ar[r]_b \ar[u]_c & \shift Z \ar[r] \ar@{=}[u] & \shift Z' \ar[u]
}$$
so as to obtain the triangle:
\begin{equation}
\label{e:thirdrow}
\xymatrix{
Z' \ar[r]^(0.4){\left[ \begin{smallmatrix} e \\ -f \end{smallmatrix} \right]} & X \oplus \shift \ubar \ar[r]^(0.6){[s\;c]} & Y \ar[r] & \shift Z'. }
\end{equation}
(see~\cite[Axiom B']{Hubery}).
Applying the octahedral axiom to the composition
$$X \stackrel{\left[^1_0\right]}{\gfl} X\oplus\shift\ubar \stackrel{[s\;c]}{\gfl}Y$$
yields the following commutative diagram whose rows and columns are triangles in $\cat$.
$$
\xymatrix{
& \ubar \bas \dreg & \ubar \bas^0 & \\
\susm Y \baseg \dr & Z \dr \bas & X \dr^s \bas & Y \baseg \\
\susm Y \dr & Z' \bas \dr & X\oplus\shift\ubar \bas \dr^{\;\;\;[s\;c]} & Y \\
& \shift\ubar \dreg & \shift \ubar &
}
$$
Note that, via an isomorphism of triangles if necessary, we may assume that the triangle in the lower row is the same as that in~\eqref{e:thirdrow}, and
thus that the morphism from $Z'$ to $\shift U$ is $-f$.
Hence, similarly, we may assume that the morphism from
$\ubar$ to $Z$ is $\shift^{-1}b$.

By construction, the morphism $[s\;c]$ admits a section $\left[^d_a\right]$
so that the triangle in the lower row splits.
Hence, $Z'$ is a summand
of $X\oplus\shift\ubar$, and thus belongs to $\cbart$ by Lemma~\ref{lemma: summands}.

We show moreover, that $Z'$ belongs to $\tbar^\perp$. 
Firstly, we note that $b$ is a right $\add \shift \tbar$-approximation of $\shift Z$. This holds since any morphism $\shift\overline{V}\fl\shift Z$ with $\overline{V}\in\add\tbar$ factors through $v$ since its composition with $u$ is zero, and thus factors through $b$ since $ba=v$. Hence $-\shift^{-1}b$ is a right
$\add \tbar$-approximation of $Z$.

Applying the functor $\cat(\tbar,-)$ to the triangle
$\ubar\fl Z\fl Z'\fl \shift\ubar$ gives $Z'\in\tbar^\perp$. By Lemma~\ref{lemma: compute perps}, $Z'$ belongs to $\add\shift T'$.
It is now easy to check (1): We constructed a triangle
$$Z'\fl X\oplus\shift\ubar \stackrel{[s\;c]}{\gfl} Y \stackrel{0}{\fl}\shift Z'$$
in $\cat$, where $Z'$ belongs to $\add\shift T'$. Hence $Z'$ belongs to $\tbar^\perp$.

We now check (2). We have
$[s\;c]\left[^d_a\right] = 1_Y$ in $\cat$.
Since $[s\;c]$ lies in $\sc$ by (1),
it is invertible in $\cat_{\sc}$ and
(2) follows.

Finally, we check the last part of the statement.
Let $\pi:X\oplus \shift \ubar \fl X$ be the first
projection. Extending $\pi$ to a triangle in $\cat$,
we have:
$$\shift \ubar \fl X\oplus \shift \ubar \stackrel{\pi}{\fl} X \stackrel{0}{\fl} \shift^2 \ubar$$.
Since $\shift \ubar\in \tbar^{\perp}$ and the zero map
factors through $T^{\perp}$, we see that $\pi\in \sc$.
Furthermore, $s\pi=[s\,0]:X\oplus \shift \ubar \fl Y$,
so $s\pi-[s\,c]=[0,-c]$ factors through $\shift\ubar$,
where $\shift\ubar$ lies in $\shift T'$. Morphisms of the form
$A\oplus V \stackrel{[1\,0]}{\gfl} A$, with $V\in\tbar^\perp$, lie in $\sc$.
Hence, as in \cite[Lemma 3.5]{BMloc1}, $L_{\sc}(s)L_{\sc}(\pi)=L_{\sc}([s\,c])$.
Since $\pi,[s\,c]$ both lie in $\sc$, their images
under $L_{\sc}$ are invertible in $\cat_{\sc}$, and it
follows that the image $L_{\sc}(s)$ is also invertible
in $\cat_\sc$, as required.
\end{proof}

For a morphism $f$ which is part of a triangle
$Z\stackrel{g}{\fl} X\stackrel{f}{\fl}Y \stackrel{h}{\fl}\shift Z$,
recall that $f$ belongs to the collection $\sc$ if and only if
$Z$ belongs to $\tbar^\perp$ and $h$ factors through $T^\perp$;
and that $f$ belongs to $\widetilde{\sc}$ if and only if
$g$ factors through $\tbar^\perp$ and $h$ factors through $T^\perp$.

\begin{theo}\label{theorem: more localisations}
There is an isomorphism of categories
$$\cat_{\sc} \simeq \cat_{\widetilde{\sc}}.$$
\end{theo}
\begin{proof}
As proved in Lemma~\ref{lemma: Stilde and SB0},
the categories $\cat_{\widetilde{\sc}}$
and $(\modl)_{\sc_{\bc,0}}$ are equivalent.
By Theorem~\ref{theorem: fbar equivalence}, the category
$(\modl)_{\sc_{\bc,0}}$ is equivalent to $\cbart/(\shift T')$.

It is easy to check that any morphism of the form
$X\oplus U \stackrel{[1\,0]}{\gfl} X$, with $U\in \tbar^{\perp}$,
lies in $\sc$.
Hence, arguing as in~\cite[Lemma 3.5]{BMloc1},
if $u,v$ are any morphisms in $\cat$ such that $v$ factors through $\tbar^{\perp}$, then
$L_{\sc}(u)=L_{\sc}(u+v)$.
It follows that $L_{\sc}$ induces a functor from
$\cbart/(\add \shift T')$ to $\cat_{\sc}$, which
we also denote by $L_{\sc}$. Since $\widetilde{\sc}$
contains $\sc$, the same argument applies to $L_{\widetilde{\sc}}$.
Furthermore, by the universal
property of localization, the left hand side of the
diagram
\[
 \xymatrix{
 & \cbart / (\add\shift T') \bg \bas \bdr^\simeq & \\
\cat_\sc \dr & \cat_{\widetilde{\sc}} \dr^\simeq & (\modl)_{\sc_{\bc,0}}
}
\]
commutes, where the functor $\cat_\sc\fl\cat_{\widetilde{\sc}}$
is the identity on objects.
The right hand triangle commutes by Lemma~\ref{lemma: Stilde and SB0}
and~\eqref{e:Fdiagram}.
It is thus enough to show that the functor  $L_\sc : \cat_\sc\fl\cbart/(\shift T')$
is an equivalence of categories.

The functor $L_\sc$ is dense by Lemma~\ref{lemma: cbart contravariantly finite}.

We next check that $L_{\sc}$ is full.
Let $X,Y$ be objects of $\cbart$ and let
$s:X\rightarrow Y$ be a morphism in $\sc$. By
part (2) of Lemma~\ref{lemma: inverse},
$sd=1_Y-ca$, where $ca$ factors through
$\tbar^{\perp}$. Arguing as above, we have that
$L_{\sc}(s)L_{\sc}(d)=L_{S}(1_Y)$,
so $L_{\sc}(d)=L_{\sc}(s)^{-1}$. It follows
that $L_{\sc}$ (on $\cbart/(\add \shift T')$) is
full.

It thus remains to prove that $L_{\sc}$ is faithful. Via the use of the functor
$\cat(\tbar,-)$ and of the category $\modlbar$, this follows from results in~\cite{BMloc1}.
Recall that we write $\overline{\Lambda}$ for the endomorphism algebra of $\tbar$ in $\cat$.
The functor $\cat(\tbar,-)$ from $\cbart/(\shift T')$ to $\modlbar$ inverts all morphisms
in $\sc$, as proved in~\cite[Lemma 2.4]{BMloc1}. By the universal property of localisations,
there is a (unique) functor $\cat_\sc \stackrel{F'}{\gfl}\modlbar$ such that
$\cat(\tbar,-) = F' L_\sc$. Assume that the image under $L_\sc$ of a morphism $f$ in $\cbart$ is zero.
Then $F'L_\sc(f)=0$ so that $\cat(\tbar,f)$ is zero in $\modlbar$. By~\cite[Lemma 2.3]{BMloc1},
this implies that $f$ factors through $\tbar^\perp$. Since $X$ is in $\cbart$,
this implies, by Lemma~\ref{lemma: factors}, that $f$ factors through $\cbart\cap\tbar^\perp$,
which is $\add\shift T'$ by Lemma~\ref{lemma: compute perps}. Therefore $f$ is zero
in $\cbart/(\shift T')$ and the functor $L_\sc$ is faithful.
\end{proof}

\begin{rk}
The reader might wonder why our proof makes a detour through the category $\modlbar$. One might think of a more direct proof as follows.
Since we have an inclusion $\sc\subseteq\widetilde{\sc}$, it is enough to prove that
every morphism in $\widetilde{\sc}$ is inverted in $\cat_\sc$.
This should easily follow from lemma~\ref{lemma: inverse}: Let
$X\stackrel{f}{\gfl}Y$ be a morphism in $\widetilde{\sc}$. Then there is a commutative diagram
\[
\xymatrix{
R_0 X \dr^{f'} \bas_{\eta_X} & R_0 Y \bas^{\eta_Y} \\
X \dr^f & Y
}
\]
where $R_0 X, R_0 Y$ are in $\cbart$ and $\eta_X,\eta_Y$ in $\sc$.
It thus only remains to be checked that the morphism $f'$ can be chosen in $\widetilde{\sc}$.
If so, Lemma~\ref{lemma: inverse} would imply that $f'$ is inverted by $L_\sc$. Since
$\eta_X$ and $\eta_Y$ are in $\sc$, $f$ would aslo be inverted by $L_\sc$.
The problem here is that even though it is easily checked that $\widetilde{\sc}$ is stable under composition,
$\widetilde{\sc}$ does not seem to satisfy the 2-out-of-3 property.
\end{rk}

\section{Examples}
\subsection{Mutating a cluster-tilting object at a loop}
\label{ss:example1}

We consider the triangulated~\cite{Keller-Orbit} orbit category
$D^\text{b}(\operatorname{mod}\rm{A}_9) / \tau^3[1]$.
Its Auslander--Reiten quiver is depicted in Figure~\ref{figure: example1}.
Copies of a fundamental domain are indicated by
dashed lines. Let $T$ be the direct sum $a\oplus b\oplus c$. Then $T$ is a cluster-tilting object with a loop at
$c$.

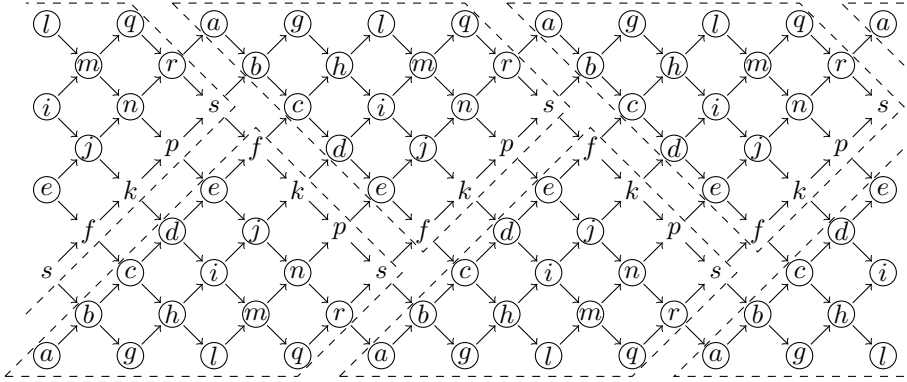
\begin{figure}
\[
\begin{tikzpicture}[scale=0.55,
baseline=(bb.base),
quivarrow/.style={black,->,shorten <=6pt, shorten >=6pt},
funddomain/.style={black,dashed}]
\newcommand{\circradius}{9pt} %

\path (0,0) node (bb) {}; 

\foreach \x in {0,1,2,3,4,5,6,7,8,9}
{
\draw[quivarrow] (2*\x+0,0) -- (2*\x+1,1);
\draw[quivarrow] (2*\x+0,2) -- (2*\x+1,1);
\draw[quivarrow] (2*\x+0,2) -- (2*\x+1,3);
\draw[quivarrow] (2*\x+0,4) -- (2*\x+1,3);
\draw[quivarrow] (2*\x+0,4) -- (2*\x+1,5);
\draw[quivarrow] (2*\x+0,6) -- (2*\x+1,5);
\draw[quivarrow] (2*\x+0,6) -- (2*\x+1,7);
\draw[quivarrow] (2*\x+0,8) -- (2*\x+1,7);

\draw[quivarrow] (2*\x+1,1) -- (2*\x+2,0);
\draw[quivarrow] (2*\x+1,1) -- (2*\x+2,2);
\draw[quivarrow] (2*\x+1,3) -- (2*\x+2,2);
\draw[quivarrow] (2*\x+1,3) -- (2*\x+2,4);
\draw[quivarrow] (2*\x+1,5) -- (2*\x+2,4);
\draw[quivarrow] (2*\x+1,5) -- (2*\x+2,6);
\draw[quivarrow] (2*\x+1,7) -- (2*\x+2,6);
\draw[quivarrow] (2*\x+1,7) -- (2*\x+2,8);
};

\draw (0,0) node {$a$};
\draw (0,0) circle(\circradius);
\draw (1,1) node {$b$};
\draw (1,1) circle(\circradius);
\draw (2,2) node {$c$};
\draw (2,2) circle(\circradius);
\draw (3,3) node {$d$};
\draw (3,3) circle(\circradius);
\draw (4,4) node {$e$};
\draw (4,4) circle(\circradius);
\draw (5,5) node {$f$};
\draw (2,0) node {$g$};
\draw (2,0) circle(\circradius);
\draw (3,1) node {$h$};
\draw (3,1) circle(\circradius);
\draw (4,2) node {$i$};
\draw (4,2) circle(\circradius);
\draw (5,3) node {$j$};
\draw (5,3) circle(\circradius);
\draw (6,4) node {$k$};
\draw (4,0) node {$l$};
\draw (4,0) circle(\circradius);
\draw (5,1) node {$m$};
\draw (5,1) circle(\circradius);
\draw (6,2) node {$n$};
\draw (6,2) circle(\circradius);
\draw (7,3) node {$p$};
\draw (6,0) node {$q$};
\draw (6,0) circle(\circradius);
\draw (7,1) node {$r$};
\draw (7,1) circle(\circradius);
\draw (8,2) node {$s$};

\begin{scope}[shift={(8,0)}]
\draw (0,0) node {$a$};
\draw (0,0) circle(\circradius);
\draw (1,1) node {$b$};
\draw (1,1) circle(\circradius);
\draw (2,2) node {$c$};
\draw (2,2) circle(\circradius);
\draw (3,3) node {$d$};
\draw (3,3) circle(\circradius);
\draw (4,4) node {$e$};
\draw (4,4) circle(\circradius);
\draw (5,5) node {$f$};
\draw (2,0) node {$g$};
\draw (2,0) circle(\circradius);
\draw (3,1) node {$h$};
\draw (3,1) circle(\circradius);
\draw (4,2) node {$i$};
\draw (4,2) circle(\circradius);
\draw (5,3) node {$j$};
\draw (5,3) circle(\circradius);
\draw (6,4) node {$k$};
\draw (4,0) node {$l$};
\draw (4,0) circle(\circradius);
\draw (5,1) node {$m$};
\draw (5,1) circle(\circradius);
\draw (6,2) node {$n$};
\draw (6,2) circle(\circradius);
\draw (7,3) node {$p$};
\draw (6,0) node {$q$};
\draw (6,0) circle(\circradius);
\draw (7,1) node {$r$};
\draw (7,1) circle(\circradius);
\draw (8,2) node {$s$};
\end{scope}

\begin{scope}[shift={(16,0)}]
\draw (0,0) node {$a$};
\draw (0,0) circle(\circradius);
\draw (1,1) node {$b$};
\draw (1,1) circle(\circradius);
\draw (2,2) node {$c$};
\draw (2,2) circle(\circradius);
\draw (3,3) node {$d$};
\draw (3,3) circle(\circradius);
\draw (4,4) node {$e$};
\draw (4,4) circle(\circradius);
\draw (2,0) node {$g$};
\draw (2,0) circle(\circradius);
\draw (3,1) node {$h$};
\draw (3,1) circle(\circradius);
\draw (4,2) node {$i$};
\draw (4,2) circle(\circradius);
\draw (4,0) node {$l$};
\draw (4,0) circle(\circradius);
\end{scope}

\draw (4,8) node {$a$};
\draw (4,8) circle(\circradius);
\draw (5,7) node {$b$};
\draw (5,7) circle(\circradius);
\draw (6,6) node {$c$};
\draw (6,6) circle(\circradius);
\draw (7,5) node {$d$};
\draw (7,5) circle(\circradius);
\draw (8,4) node {$e$};
\draw (8,4) circle(\circradius);
\draw (9,3) node {$f$};
\draw (6,8) node {$g$};
\draw (6,8) circle(\circradius);
\draw (7,7) node {$h$};
\draw (7,7) circle(\circradius);
\draw (8,6) node {$i$};
\draw (8,6) circle(\circradius);
\draw (9,5) node {$j$};
\draw (9,5) circle(\circradius);
\draw (10,4) node {$k$};
\draw (8,8) node {$l$};
\draw (8,8) circle(\circradius);
\draw (9,7) node {$m$};
\draw (9,7) circle(\circradius);
\draw (10,6) node {$n$};
\draw (10,6) circle(\circradius);
\draw (11,5) node {$p$};
\draw (10,8) node {$q$};
\draw (10,8) circle(\circradius);
\draw (11,7) node {$r$};
\draw (11,7) circle(\circradius);
\draw (12,6) node {$s$};

\begin{scope}[shift={(8,0)}]
\draw (4,8) node {$a$};
\draw (4,8) circle(\circradius);
\draw (5,7) node {$b$};
\draw (5,7) circle(\circradius);
\draw (6,6) node {$c$};
\draw (6,6) circle(\circradius);
\draw (7,5) node {$d$};
\draw (7,5) circle(\circradius);
\draw (8,4) node {$e$};
\draw (8,4) circle(\circradius);
\draw (9,3) node {$f$};
\draw (6,8) node {$g$};
\draw (6,8) circle(\circradius);
\draw (7,7) node {$h$};
\draw (7,7) circle(\circradius);
\draw (8,6) node {$i$};
\draw (8,6) circle(\circradius);
\draw (9,5) node {$j$};
\draw (9,5) circle(\circradius);
\draw (10,4) node {$k$};
\draw (8,8) node {$l$};
\draw (8,8) circle(\circradius);
\draw (9,7) node {$m$};
\draw (9,7) circle(\circradius);
\draw (10,6) node {$n$};
\draw (10,6) circle(\circradius);
\draw (11,5) node {$p$};
\draw (10,8) node {$q$};
\draw (10,8) circle(\circradius);
\draw (11,7) node {$r$};
\draw (11,7) circle(\circradius);
\draw (12,6) node {$s$};
\end{scope}

\draw (20,8) node {$a$};
\draw (20,8) circle(\circradius);

\begin{scope}[shift={(-8,0)}]
\draw (8,4) node {$e$};
\draw (8,4) circle(\circradius);
\draw (9,3) node {$f$};
\draw (8,6) node {$i$};
\draw (8,6) circle(\circradius);
\draw (9,5) node {$j$};
\draw (9,5) circle(\circradius);
\draw (10,4) node {$k$};
\draw (8,8) node {$l$};
\draw (8,8) circle(\circradius);
\draw (9,7) node {$m$};
\draw (9,7) circle(\circradius);
\draw (10,6) node {$n$};
\draw (10,6) circle(\circradius);
\draw (11,5) node {$p$};
\draw (10,8) node {$q$};
\draw (10,8) circle(\circradius);
\draw (11,7) node {$r$};
\draw (11,7) circle(\circradius);
\draw (12,6) node {$s$};
\end{scope}

\draw (0,2) node {$s$};

\draw[funddomain] (-1,-0.5) -- (5,5.5) -- (8.5,2) -- (6,-0.5) -- (-1,-0.5);

\draw[funddomain] (7,-0.5) -- (13,5.5) -- (16.5,2) -- (14,-0.5) -- (7,-0.5);

\draw[funddomain] (20.5,-0.5) -- (15,-0.5) -- (20.5,5);

\draw[funddomain] (3,8.5) -- (9,2.5) -- (12.5,6) -- (10,8.5) -- (3,8.5);

\draw[funddomain] (11,8.5) -- (17,2.5) -- (20.5,6) -- (18,8.5) -- (11,8.5);

\draw[funddomain] (20.5,8.5) -- (19,8.5) -- (20.5,7);

\draw[funddomain] (-0.5,1) -- (4.5,6) -- (2,8.5) -- (-0.5,8.5);

\end{tikzpicture}
\]
 \caption{The AR-quiver of $D^\text{b}(\operatorname{mod}\rm{A}_9) / \tau^3[1]$; Example~\ref{ss:example1}}
 \label{figure: example1}
\end{figure}

Let $T'$ be the cluster-tilting object obtained
by mutating $T$ at $c$. Since there is a triangle
$s \fl b\oplus b \fl c \fl$, we have
$T' = a\oplus b\oplus s$.
The indecomposable objects lying in $\cbart$ are encircled (recall that $\cbart = T\ast\shift\tbar = \tbar\ast\shift T'$).
The objects $a,b,c$ (respectively, $n,q,r$) belong to $\cbart$ since they are in $\add T$ (respectively $\add\shift T'$). The remaining encircled objects are in
$\cbart$ since there are triangles $a \fl c \fl d \fl$;
$b\fl c\fl e\fl$;
$a\fl b\fl g\fl$;
$a\fl c\fl h\fl$;
$a\oplus a\fl c\fl i\fl$;
$a\oplus b\fl c\fl j\fl$;
$b\fl c\fl l$ and
$a\oplus b\fl c\fl m\fl$.
The other four indecomposable objects are not in
$\cbart$ since $s$ is the shift of $c$, and there are triangles $c\fl c\fl f\fl$, $a\oplus c\fl c\fl k\fl$ and
$b\oplus c\fl c\fl p\fl$.

By Theorem~\ref{theorem: equivalence of categories}, the categories $\cbart/(\shift T')$ and
$\tau\cbart/(\tau T)$ are equivalent. These two categories are illustrated in
Figure~\ref{figure: example1-equivalence}.

\begin{figure}
\[
\begin{tikzpicture}[xscale=0.35, yscale=0.5,
baseline=(bb.base),
quivarrow/.style={black,->,shorten <=6pt, shorten >=6pt},
funddomain/.style={black,dashed}]
\newcommand{\circradius}{9pt} %

\path (0,0) node (bb) {}; 

\draw (0,0) node {$a$};
\draw (1,1) node {$b$};
\draw (2,2) node {$c$};
\draw (3,3) node {$d$};
\draw (4,4) node {$e$};
\draw (2,0) node {$g$};
\draw (3,1) node {$h$};
\draw (4,2) node {$i$};
\draw (5,3) node {$j$};
\draw (4,0) node {$l$};
\draw (5,1) node {$m$};

\draw[gray] (6,2) node {$n$};
\draw[gray] (6,0) node {$q$};
\draw[gray] (7,1) node {$r$};

\draw[quivarrow] (2,6) -- (4,4);
\draw[quivarrow] (1,5) -- (3,3);

\draw[quivarrow] (0,0) -- (1,1);
\draw[quivarrow] (1,1) -- (2,2);
\draw[quivarrow] (2,2) -- (3,3);
\draw[quivarrow] (3,3) -- (4,4);
\draw[quivarrow] (2,0) -- (3,1);
\draw[quivarrow] (3,1) -- (4,2);
\draw[quivarrow] (4,2) -- (5,3);
\draw[quivarrow] (4,0) -- (5,1);

\draw[quivarrow] (1,1) -- (2,0);
\draw[quivarrow] (2,2) -- (3,1);
\draw[quivarrow] (3,3) -- (4,2);
\draw[quivarrow] (4,4) -- (5,3);
\draw[quivarrow] (3,1) -- (4,0);
\draw[quivarrow] (4,2) -- (5,1);

\draw[quivarrow] (4,4) -- (6,6);
\draw[quivarrow] (5,3) -- (7,5);

\draw (4,8) node {$a$};
\draw (5,7) node {$b$};
\draw (6,6) node {$c$};
\draw (7,5) node {$d$};
\draw (8,4) node {$e$};
\draw (6,8) node {$g$};
\draw (7,7) node {$h$};
\draw (8,6) node {$i$};
\draw (9,5) node {$j$};
\draw (8,8) node {$l$};
\draw (9,7) node {$m$};

\draw[quivarrow] (4,8) -- (5,7);
\draw[quivarrow] (5,7) -- (6,6);
\draw[quivarrow] (6,6) -- (7,5);
\draw[quivarrow] (7,5) -- (8,4);
\draw[quivarrow] (6,8) -- (7,7);
\draw[quivarrow] (7,7) -- (8,6);
\draw[quivarrow] (8,6) -- (9,5);
\draw[quivarrow] (8,8) -- (9,7);

\draw[quivarrow] (5,7) -- (6,8);
\draw[quivarrow] (6,6) -- (7,7);
\draw[quivarrow] (7,5) -- (8,6);
\draw[quivarrow] (8,4) -- (9,5);
\draw[quivarrow] (7,7) -- (8,8);
\draw[quivarrow] (8,6) -- (9,7);

\draw[quivarrow] (8,4) -- (10,2);
\draw[quivarrow] (9,5) -- (11,3);

\draw[gray] (10,6) node {$n$};
\draw[gray] (10,8) node {$q$};
\draw[gray] (11,7) node {$r$};

\begin{scope}[shift={(8,0)}]
\draw (0,0) node {$a$};
\draw (1,1) node {$b$};
\draw (2,2) node {$c$};
\draw (3,3) node {$d$};
\draw (4,4) node {$e$};
\draw (2,0) node {$g$};
\draw (3,1) node {$h$};
\draw (4,2) node {$i$};
\draw (5,3) node {$j$};
\draw (4,0) node {$l$};
\draw (5,1) node {$m$};

\draw[gray] (6,2) node {$n$};
\draw[gray] (6,0) node {$q$};
\draw[gray] (7,1) node {$r$};

\draw[quivarrow] (0,0) -- (1,1);
\draw[quivarrow] (1,1) -- (2,2);
\draw[quivarrow] (2,2) -- (3,3);
\draw[quivarrow] (3,3) -- (4,4);
\draw[quivarrow] (2,0) -- (3,1);
\draw[quivarrow] (3,1) -- (4,2);
\draw[quivarrow] (4,2) -- (5,3);
\draw[quivarrow] (4,0) -- (5,1);

\draw[quivarrow] (1,1) -- (2,0);
\draw[quivarrow] (2,2) -- (3,1);
\draw[quivarrow] (3,3) -- (4,2);
\draw[quivarrow] (4,4) -- (5,3);
\draw[quivarrow] (3,1) -- (4,0);
\draw[quivarrow] (4,2) -- (5,1);

\draw[quivarrow] (4,4) -- (6,6);
\draw[quivarrow] (5,3) -- (7,5);
\end{scope}

\end{tikzpicture}
\quad\quad\,\,
\begin{tikzpicture}[xscale=0.35, yscale=0.5,
baseline=(bb.base),
quivarrow/.style={black,->,shorten <=6pt, shorten >=6pt},
funddomain/.style={black,dashed}]
\newcommand{\circradius}{9pt} %

\path (0,0) node (bb) {}; 

\draw (0,0) node {$a$};
\draw (1,1) node {$b$};
\draw (2,2) node {$c$};
\draw (3,3) node {$d$};
\draw (5,5) node {$f$};
\draw (2,0) node {$g$};
\draw (3,1) node {$h$};
\draw (4,2) node {$i$};
\draw (6,4) node {$k$};
\draw (4,0) node {$l$};
\draw (5,1) node {$m$};

\draw[gray] (8,2) node {$s$};
\draw[gray] (6,0) node {$q$};
\draw[gray] (7,1) node {$r$};

\draw[quivarrow] (2,4) -- (3,3);
\draw[quivarrow] (1,3) -- (2,2);

\draw[quivarrow] (0,0) -- (1,1);
\draw[quivarrow] (1,1) -- (2,2);
\draw[quivarrow] (2,2) -- (3,3);
\draw[quivarrow] (3,3) -- (5,5);
\draw[quivarrow] (2,0) -- (3,1);
\draw[quivarrow] (3,1) -- (4,2);
\draw[quivarrow] (4,2) -- (6,4);
\draw[quivarrow] (4,0) -- (5,1);

\draw[quivarrow] (1,1) -- (2,0);
\draw[quivarrow] (2,2) -- (3,1);
\draw[quivarrow] (3,3) -- (4,2);
\draw[quivarrow] (5,5) -- (6,4);
\draw[quivarrow] (3,1) -- (4,0);
\draw[quivarrow] (4,2) -- (5,1);

\draw[quivarrow] (5,5) -- (6,6);
\draw[quivarrow] (6,4) -- (7,5);

\draw (4,8) node {$a$};
\draw (5,7) node {$b$};
\draw (6,6) node {$c$};
\draw (7,5) node {$d$};
\draw (9,3) node {$f$};
\draw (6,8) node {$g$};
\draw (7,7) node {$h$};
\draw (8,6) node {$i$};
\draw (10,4) node {$k$};
\draw (8,8) node {$l$};
\draw (9,7) node {$m$};

\draw[gray] (12,6) node {$s$};
\draw[gray] (10,8) node {$q$};
\draw[gray] (11,7) node {$r$};

\draw[quivarrow] (4,8) -- (5,7);
\draw[quivarrow] (5,7) -- (6,6);
\draw[quivarrow] (6,6) -- (7,5);
\draw[quivarrow] (7,5) -- (9,3);
\draw[quivarrow] (6,8) -- (7,7);
\draw[quivarrow] (7,7) -- (8,6);
\draw[quivarrow] (8,6) -- (10,4);
\draw[quivarrow] (8,8) -- (9,7);

\draw[quivarrow] (5,7) -- (6,8);
\draw[quivarrow] (6,6) -- (7,7);
\draw[quivarrow] (7,5) -- (8,6);
\draw[quivarrow] (9,3) -- (10,4);
\draw[quivarrow] (7,7) -- (8,8);
\draw[quivarrow] (8,6) -- (9,7);

\draw[quivarrow] (7,5) -- (9,3);
\draw[quivarrow] (8,6) -- (10,4);

\draw[quivarrow] (9,3) -- (10,2);
\draw[quivarrow] (10,4) -- (11,3);

\begin{scope}[shift={(8,0)}]
\draw (0,0) node {$a$};
\draw (1,1) node {$b$};
\draw (2,2) node {$c$};
\draw (3,3) node {$d$};
\draw (5,5) node {$f$};
\draw (2,0) node {$g$};
\draw (3,1) node {$h$};
\draw (4,2) node {$i$};
\draw (6,4) node {$k$};
\draw (4,0) node {$l$};
\draw (5,1) node {$m$};

\draw[gray] (8,2) node {$s$};
\draw[gray] (6,0) node {$q$};
\draw[gray] (7,1) node {$r$};

\draw[quivarrow] (0,0) -- (1,1);
\draw[quivarrow] (1,1) -- (2,2);
\draw[quivarrow] (2,2) -- (3,3);
\draw[quivarrow] (3,3) -- (5,5);
\draw[quivarrow] (2,0) -- (3,1);
\draw[quivarrow] (3,1) -- (4,2);
\draw[quivarrow] (4,2) -- (6,4);
\draw[quivarrow] (4,0) -- (5,1);

\draw[quivarrow] (1,1) -- (2,0);
\draw[quivarrow] (2,2) -- (3,1);
\draw[quivarrow] (3,3) -- (4,2);
\draw[quivarrow] (5,5) -- (6,4);
\draw[quivarrow] (3,1) -- (4,0);
\draw[quivarrow] (4,2) -- (5,1);

\draw[quivarrow] (5,5) -- (6,6);
\draw[quivarrow] (6,4) -- (7,5);
\end{scope}

\end{tikzpicture}
\]
\caption{The categories $\cbart/(\shift T')$ and $\tau\cbart/(\tau T)$; Example~\ref{ss:example1}}
\label{figure: example1-equivalence}
\end{figure}
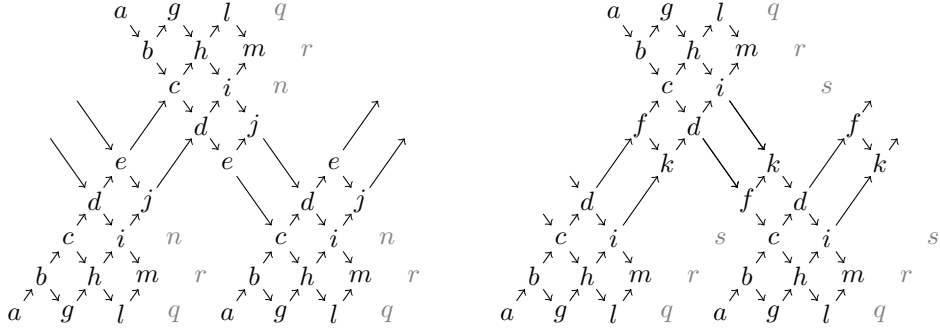

\subsection{Mutating a rigid object at a loop}
\label{ss:example2}

In the same category
$D^\text{b}(\operatorname{mod}\rm{A}_9) / \tau^3[1]$
as in the previous example, we now consider the rigid object $T = a\oplus c$. There is a triangle
$n \fl a\oplus a \fl c \fl$
(see Figure~\ref{figure: example2}),
so that we can choose $T'$ to be $a\oplus n$.
Indecomposable objects in $\cbart$ are encircled.
The shift of $T'$ is $i\oplus q$. Deleting either vertices labelled $a$ and $c$,
or vertices $i$ and $q$, in the encircled part of the
AR-quiver yields isomorphic quivers,
as depicted in Figure \ref{figure: example2-equivalence}.

\begin{figure}
\[
\begin{tikzpicture}[scale=0.55,
baseline=(bb.base),
quivarrow/.style={black,->,shorten <=6pt, shorten >=6pt},
funddomain/.style={black,dashed}]
\newcommand{\circradius}{9pt} %

\path (0,0) node (bb) {}; 

\foreach \x in {0,1,2,3,4,5,6,7,8,9}
{
\draw[quivarrow] (2*\x+0,0) -- (2*\x+1,1);
\draw[quivarrow] (2*\x+0,2) -- (2*\x+1,1);
\draw[quivarrow] (2*\x+0,2) -- (2*\x+1,3);
\draw[quivarrow] (2*\x+0,4) -- (2*\x+1,3);
\draw[quivarrow] (2*\x+0,4) -- (2*\x+1,5);
\draw[quivarrow] (2*\x+0,6) -- (2*\x+1,5);
\draw[quivarrow] (2*\x+0,6) -- (2*\x+1,7);
\draw[quivarrow] (2*\x+0,8) -- (2*\x+1,7);

\draw[quivarrow] (2*\x+1,1) -- (2*\x+2,0);
\draw[quivarrow] (2*\x+1,1) -- (2*\x+2,2);
\draw[quivarrow] (2*\x+1,3) -- (2*\x+2,2);
\draw[quivarrow] (2*\x+1,3) -- (2*\x+2,4);
\draw[quivarrow] (2*\x+1,5) -- (2*\x+2,4);
\draw[quivarrow] (2*\x+1,5) -- (2*\x+2,6);
\draw[quivarrow] (2*\x+1,7) -- (2*\x+2,6);
\draw[quivarrow] (2*\x+1,7) -- (2*\x+2,8);
};

\draw (0,0) node {$a$};
\draw (0,0) circle(\circradius);
\draw (1,1) node {$b$};
\draw (2,2) node {$c$};
\draw (2,2) circle(\circradius);
\draw (3,3) node {$d$};
\draw (3,3) circle(\circradius);
\draw (4,4) node {$e$};
\draw (5,5) node {$f$};
\draw (2,0) node {$g$};
\draw (3,1) node {$h$};
\draw (3,1) circle(\circradius);
\draw (4,2) node {$i$};
\draw (4,2) circle(\circradius);
\draw (5,3) node {$j$};
\draw (6,4) node {$k$};
\draw (4,0) node {$l$};
\draw (5,1) node {$m$};
\draw (6,2) node {$\mathbf{n}$};
\draw (7,3) node {$p$};
\draw (6,0) node {$q$};
\draw (6,0) circle(\circradius);
\draw (7,1) node {$r$};
\draw (8,2) node {$s$};

\begin{scope}[shift={(8,0)}]
\draw (0,0) node {$a$};
\draw (0,0) circle(\circradius);
\draw (1,1) node {$b$};
\draw (2,2) node {$c$};
\draw (2,2) circle(\circradius);
\draw (3,3) node {$d$};
\draw (3,3) circle(\circradius);
\draw (4,4) node {$e$};
\draw (5,5) node {$f$};
\draw (2,0) node {$g$};
\draw (3,1) node {$h$};
\draw (3,1) circle(\circradius);
\draw (4,2) node {$i$};
\draw (4,2) circle(\circradius);
\draw (5,3) node {$j$};
\draw (6,4) node {$k$};
\draw (4,0) node {$l$};
\draw (5,1) node {$m$};
\draw (6,2) node {$\mathbf{n}$};
\draw (7,3) node {$p$};
\draw (6,0) node {$q$};
\draw (6,0) circle(\circradius);
\draw (7,1) node {$r$};
\draw (8,2) node {$s$};
\end{scope}

\begin{scope}[shift={(16,0)}]
\draw (0,0) node {$a$};
\draw (0,0) circle(\circradius);
\draw (1,1) node {$b$};
\draw (2,2) node {$c$};
\draw (2,2) circle(\circradius);
\draw (3,3) node {$d$};
\draw (3,3) circle(\circradius);
\draw (4,4) node {$e$};
\draw (2,0) node {$g$};
\draw (3,1) node {$h$};
\draw (3,1) circle(\circradius);
\draw (4,2) node {$i$};
\draw (4,2) circle(\circradius);
\draw (4,0) node {$l$};
\end{scope}

\draw (4,8) node {$a$};
\draw (4,8) circle(\circradius);
\draw (5,7) node {$b$};
\draw (6,6) node {$c$};
\draw (6,6) circle(\circradius);
\draw (7,5) node {$d$};
\draw (7,5) circle(\circradius);
\draw (8,4) node {$e$};
\draw (9,3) node {$f$};
\draw (6,8) node {$g$};
\draw (7,7) node {$h$};
\draw (7,7) circle(\circradius);
\draw (8,6) node {$i$};
\draw (8,6) circle(\circradius);
\draw (9,5) node {$j$};
\draw (10,4) node {$k$};
\draw (8,8) node {$l$};
\draw (9,7) node {$m$};
\draw (10,6) node {$\mathbf{n}$};
\draw (11,5) node {$p$};
\draw (10,8) node {$q$};
\draw (10,8) circle(\circradius);
\draw (11,7) node {$r$};
\draw (12,6) node {$s$};

\begin{scope}[shift={(8,0)}]
\draw (4,8) node {$a$};
\draw (4,8) circle(\circradius);
\draw (5,7) node {$b$};
\draw (6,6) node {$c$};
\draw (6,6) circle(\circradius);
\draw (7,5) node {$d$};
\draw (7,5) circle(\circradius);
\draw (8,4) node {$e$};
\draw (9,3) node {$f$};
\draw (6,8) node {$g$};
\draw (7,7) node {$h$};
\draw (7,7) circle(\circradius);
\draw (8,6) node {$i$};
\draw (8,6) circle(\circradius);
\draw (9,5) node {$j$};
\draw (10,4) node {$k$};
\draw (8,8) node {$l$};
\draw (9,7) node {$m$};
\draw (10,6) node {$\mathbf{n}$};
\draw (11,5) node {$p$};
\draw (10,8) node {$q$};
\draw (10,8) circle(\circradius);
\draw (11,7) node {$r$};
\draw (12,6) node {$s$};
\end{scope}

\draw (20,8) node {$a$};
\draw (20,8) circle(\circradius);

\begin{scope}[shift={(-8,0)}]
\draw (8,4) node {$e$};
\draw (9,3) node {$f$};
\draw (8,6) node {$i$};
\draw (8,6) circle(\circradius);
\draw (9,5) node {$j$};
\draw (10,4) node {$k$};
\draw (8,8) node {$l$};
\draw (9,7) node {$m$};
\draw (10,6) node {$\mathbf{n}$};
\draw (11,5) node {$p$};
\draw (10,8) node {$q$};
\draw (10,8) circle(\circradius);
\draw (11,7) node {$r$};
\draw (12,6) node {$s$};
\end{scope}

\draw (0,2) node {$s$};

\draw[funddomain] (-1,-0.5) -- (5,5.5) -- (8.5,2) -- (6,-0.5) -- (-1,-0.5);

\draw[funddomain] (7,-0.5) -- (13,5.5) -- (16.5,2) -- (14,-0.5) -- (7,-0.5);

\draw[funddomain] (20.5,-0.5) -- (15,-0.5) -- (20.5,5);

\draw[funddomain] (3,8.5) -- (9,2.5) -- (12.5,6) -- (10,8.5) -- (3,8.5);

\draw[funddomain] (11,8.5) -- (17,2.5) -- (20.5,6) -- (18,8.5) -- (11,8.5);

\draw[funddomain] (20.5,8.5) -- (19,8.5) -- (20.5,7);

\draw[funddomain] (-0.5,1) -- (4.5,6) -- (2,8.5) -- (-0.5,8.5);

\end{tikzpicture}
\]

\caption{The AR-quiver of $D^\text{b}(\operatorname{mod}\rm{A}_9) / \tau^3[1]$; Example~\ref{ss:example2}}
\label{figure: example2}
\end{figure}

\begin{figure}
\[
\begin{tikzpicture}[xscale=0.35, yscale=0.5,
baseline=(bb.base),
quivarrow/.style={black,->,shorten <=6pt, shorten >=6pt},
funddomain/.style={black,dashed}]
\newcommand{\circradius}{9pt} %

\path (0,0) node (bb) {}; 

\draw (0,0) node {$a$};
\draw (2,2) node {$c$};
\draw (3,3) node {$d$};
\draw (3,1) node {$h$};
\draw[gray] (4,2) node {$i$};
\draw[gray] (6,0) node {$q$};

\draw[quivarrow] (0,0) -- (2,2);
\draw[quivarrow] (2,2) -- (3,3);
\draw[quivarrow] (2,2) -- (3,1);
\draw[quivarrow] (-1,5) -- (2,2);

\begin{scope}[shift={(8,0)}]
\draw (0,0) node {$a$};
\draw (2,2) node {$c$};
\draw (3,3) node {$d$};
\draw (3,1) node {$h$};
\draw[gray] (4,2) node {$i$};
\draw[gray] (6,0) node {$q$};

\draw[quivarrow] (0,0) -- (2,2);
\draw[quivarrow] (2,2) -- (3,3);
\draw[quivarrow] (2,2) -- (3,1);
\draw[quivarrow] (-1,5) -- (2,2);
\draw[quivarrow] (3,3) -- (6,6);
\end{scope}

\draw (4,8) node {$a$};
\draw (6,6) node {$c$};
\draw (7,5) node {$d$};
\draw (7,7) node {$h$};
\draw[gray] (8,6) node {$i$};
\draw[gray] (10,8) node {$q$};

\draw[quivarrow] (3,3) -- (6,6);
\draw[quivarrow] (4,8) -- (6,6);
\draw[quivarrow] (6,6) -- (7,5);
\draw[quivarrow] (6,6) -- (7,7);

\end{tikzpicture}
\quad\quad\quad\quad\quad
\begin{tikzpicture}[xscale=0.35, yscale=0.5,
baseline=(bb.base),
quivarrow/.style={black,->,shorten <=6pt, shorten >=6pt},
funddomain/.style={black,dashed}]
\newcommand{\circradius}{9pt} %

\path (0,0) node (bb) {}; 

\draw[gray] (0,0) node {$a$};
\draw[gray] (2,2) node {$c$};
\draw (3,3) node {$d$};
\draw (3,1) node {$h$};
\draw (4,2) node {$i$};
\draw (6,0) node {$q$};

\draw[quivarrow] (3,3) -- (4,2);
\draw[quivarrow] (3,1) -- (4,2);
\draw[quivarrow] (4,2) -- (6,0);
\draw[quivarrow] (0,6) -- (3,3);
\draw[quivarrow] (4,2) -- (7,5);

\begin{scope}[shift={(8,0)}]
\draw[gray] (0,0) node {$a$};
\draw[gray] (2,2) node {$c$};
\draw (3,3) node {$d$};
\draw (3,1) node {$h$};
\draw (4,2) node {$i$};
\draw (6,0) node {$q$};

\draw[quivarrow] (3,3) -- (4,2);
\draw[quivarrow] (3,1) -- (4,2);
\draw[quivarrow] (4,2) -- (6,0);
\draw[quivarrow] (0,6) -- (3,3);
\draw[quivarrow] (4,2) -- (7,5);
\end{scope}

\draw[gray] (4,8) node {$a$};
\draw[gray] (6,6) node {$c$};
\draw (7,5) node {$d$};
\draw (7,7) node {$h$};
\draw (8,6) node {$i$};
\draw (10,8) node {$q$};

\draw[quivarrow] (7,7) -- (8,6);
\draw[quivarrow] (7,5) -- (8,6);
\draw[quivarrow] (8,6) -- (10,8);

\end{tikzpicture}
\]
\caption{The categories $\cbart/(\shift T')$ and $\cbart/(T)$; Example~\ref{ss:example2}}
\label{figure: example2-equivalence}
\end{figure}

\subsection{Mutating at a non-indecomposable summand}
\label{ss:example3}

In this example, we consider the triangulated~\cite{Keller-Orbit} orbit category
$D^\text{b}(\operatorname{mod}\rm{A}_5) / \tau^{-2}[1]$. It was shown in~\cite{BMRRT} that this category
is a Krull--Schmidt, Hom-finite category with a Serre functor and its Auslander--Reiten quiver is depicted in Figure~\ref{figure: example3}.
As for the previous examples, we have not drawn the arrows. The two subquivers in inside the dotted boxes
have to be identified so as to match the two copies of $a,b,c$ and $d$.
We choose rigid objects is $T = a\oplus b\oplus c\oplus d$, and $T' = a\oplus b\oplus c'\oplus d'$. Indecomposable objects in $\cbart$ are encircled.
The indecomposable objects labelled $e,f,g,h$ are the shifts of $a,b,c',d'$, respectively.
As predicted by Theorem~\ref{theorem: equivalence of categories}, one obtains isomorphic quivers by deleting either vertices $a,b,c,d$ or vertices $e,f,g,h$.

\begin{figure}
\[
\begin{tikzpicture}[scale=0.55,
baseline=(bb.base),
quivarrow/.style={black,->,shorten <=6pt, shorten >=6pt},
funddomain/.style={black,dashed}]
\newcommand{\circradius}{9pt} %

\path (0,0) node (bb) {}; 

\foreach \x in {0,1,2,3,4,5,6,7}
{
\draw[quivarrow] (2*\x+0,0) -- (2*\x+1,1);
\draw[quivarrow] (2*\x+0,2) -- (2*\x+1,1);
\draw[quivarrow] (2*\x+0,2) -- (2*\x+1,3);
\draw[quivarrow] (2*\x+0,4) -- (2*\x+1,3);

\draw[quivarrow] (2*\x+1,1) -- (2*\x+2,0);
\draw[quivarrow] (2*\x+1,1) -- (2*\x+2,2);
\draw[quivarrow] (2*\x+1,3) -- (2*\x+2,2);
\draw[quivarrow] (2*\x+1,3) -- (2*\x+2,4);
};

\draw (0,0) node {$\bullet$};
\draw (4,0) node {$\bullet$};
\draw (6,0) node {$\bullet$};
\draw (8,0) node {$\bullet$};
\draw (10,0) node {$\bullet$};
\draw (12,0) node {$\bullet$};
\draw (14,0) node {$\bullet$};
\draw (1,1) node {$\bullet$};
\draw (5,1) node {$\bullet$};
\draw (9,1) node {$\bullet$};
\draw (13,1) node {$\bullet$};
\draw (0,2) node {$\bullet$};
\draw (4,2) node {$\bullet$};
\draw (10,2) node {$\bullet$};
\draw (14,2) node {$\bullet$};
\draw (3,3) node {$\bullet$};
\draw (11,3) node {$\bullet$};
\draw (15,3) node {$\bullet$};
\draw (0,4) node {$\bullet$};
\draw (2,4) node {$\bullet$};
\draw (4,4) node {$\bullet$};
\draw (8,4) node {$\bullet$};
\draw (10,4) node {$\bullet$};
\draw (16,4) node {$\bullet$};

\draw (1,3) node {$c'$};
\draw (2,2) node {$d'$};
\draw (2,0) node {$a$};
\draw (2,0) circle(\circradius);
\draw (3,1) node {$b$};
\draw (3,1) circle(\circradius);
\draw (5,3) node {$c$};
\draw (5,3) circle(\circradius);
\draw (6,4) node {$d$};
\draw (6,4) circle(\circradius);
\draw (7,1) node {$g$};
\draw (7,1) circle(\circradius);
\draw (8,2) node {$h$};
\draw (8,2) circle(\circradius);
\draw (9,3) node {$f$};
\draw (9,3) circle(\circradius);
\draw (11,1) node {$c'$};
\draw (12,2) node {$d'$};
\draw (13,3) node {$b$};
\draw (13,3) circle(\circradius);
\draw (12,4) node {$a$};
\draw (12,4) circle(\circradius);
\draw (15,1) node {$c$};
\draw (15,1) circle(\circradius);
\draw (16,0) node {$d$};
\draw (16,0) circle(\circradius);

\draw (4,0) circle(\circradius);
\draw (4,0) node {$\bullet$};
\draw (6,2) circle(\circradius);
\draw (6,2) node {$\bullet$};
\draw (7,3) circle(\circradius);
\draw (7,3) node {$\bullet$};
\draw (14,4) circle(\circradius);
\draw (14,4) node {$\bullet$};
\draw (16,2) circle(\circradius);
\draw (16,2) node {$\bullet$};

\draw[dotted] (1.2,0) -- (6,4.8) -- (6.8,4) -- (2,-0.8) -- (1.2,0);

\draw[dotted] (16.8,0) -- (12,4.8) -- (11.2,4) -- (16,-0.8) -- (16.8,0);

\end{tikzpicture}
\]
\caption{The AR-quiver of $D^\text{b}(\operatorname{mod}\rm{A}_5) / \tau^{-2}[1]$; Example~\ref{ss:example3}}
\label{figure: example3}
\end{figure}

\newcommand{\etalchar}[1]{$^{#1}$}


\begin{thebibliography}{BMR{\etalchar{+}}06}

\bibitem[Ami07]{Amiot-TriangulatedCategories}
Claire Amiot.
\newblock On the structure of triangulated categories with finitely many
  indecomposables.
\newblock {\em Bull. Soc. Math. France}, 135(3):435--474, 2007.

\bibitem[Ami09]{Amiot-ClusterCategories}
Claire Amiot.
\newblock Cluster categories for algebras of global dimension 2 and quivers
  with potential.
\newblock {\em Ann. Inst. Fourier (Grenoble)}, 59(6):2525--2590, 2009.

\bibitem[Bel13]{Beligiannis}
A.~D. Beligiannis.
\newblock Rigid objects, triangulated subfactors and abelian localizations.
\newblock {\em Math. Zeit.}, 274(3):841--883, 2013.

\bibitem[BIKR08]{BIKR}
Igor Burban, Osamu Iyama, Bernhard Keller, and Idun Reiten.
\newblock Cluster tilting for one-dimensional hypersurface singularities.
\newblock {\em Adv. Math.}, 217(6):2443--2484, 2008.

\bibitem[BKL08]{BKL}
M.~Barot, D.~Kussin, and H.~Lenzing.
\newblock The {G}rothendieck group of a cluster category.
\newblock {\em J. Pure Appl. Algebra}, 212(1):33--46, 2008.

\bibitem[BM12]{BMloc2}
Aslak~Bakke Buan and Robert~J. Marsh.
\newblock From triangulated categories to module categories via localization
  {II}: calculus of fractions.
\newblock {\em J. Lond. Math. Soc. (2)}, 86(1):152--170, 2012.

\bibitem[BM13]{BMloc1}
Aslak~Bakke Buan and Robert~J. Marsh.
\newblock From triangulated categories to module categories via localisation.
\newblock {\em Trans. Amer. Math. Soc.}, 365(6):2845--2861, 2013.

\bibitem[BMR{\etalchar{+}}06]{BMRRT}
Aslak~Bakke Buan, Robert Marsh, Markus Reineke, Idun Reiten, and Gordana
  Todorov.
\newblock Tilting theory and cluster combinatorics.
\newblock {\em Adv. Math.}, 204(2):572--618, 2006.

\bibitem[BMR07]{BMR-CTA}
Aslak~Bakke Buan, Robert~J. Marsh, and Idun Reiten.
\newblock Cluster-tilted algebras.
\newblock {\em Trans. Amer. Math. Soc.}, 359(1):323--332 (electronic), 2007.

\bibitem[BMV10]{BMV}
Aslak~Bakke Buan, Robert~J. Marsh, and Dagfinn~F. Vatne.
\newblock Cluster structures from 2-{C}alabi-{Y}au categories with loops.
\newblock {\em Math. Z.}, 265(4):951--970, 2010.

\bibitem[GLS]{GLS-ClusterAlgebraStructures}
C.~Gei{\ss}, B.~Leclerc, and J.~Schr{\"o}er.
\newblock Cluster algebra structures and semicanonical bases for unipotent
  groups.
\newblock {\em preprint arXiv:math/0703039 [math.RT]}.

\bibitem[Guo11]{Guo-HigherClusterCategories}
Lingyan Guo.
\newblock Cluster tilting objects in generalized higher cluster categories.
\newblock {\em J. Pure Appl. Algebra}, 215(9):2055--2071, 2011.

\bibitem[HJ12]{HJ-Ainfinity}
Thorsten Holm and Peter J{\o}rgensen.
\newblock On a cluster category of infinite {D}ynkin type, and the relation to
  triangulations of the infinity-gon.
\newblock {\em Math. Z.}, 270(1-2):277--295, 2012.

\bibitem[HJ13]{HJ-HigherClusterCategories}
Thorsten Holm and Peter J{\o}rgensen.
\newblock Realizing higher cluster categories of {D}ynkin type as stable module
  categories.
\newblock {\em Q. J. Math.}, 64(2):409--435, 2013.

\bibitem[H]{Hubery}
A. Hubery. Notes on the octahedral axiom.
Available from
http://www1.maths.leeds.ac.uk/$\sim$ahubery/LectureNotes.html

\bibitem[IY08]{IY}
Osamu Iyama and Yuji Yoshino.
\newblock Mutation in triangulated categories and rigid {C}ohen-{M}acaulay
  modules.
\newblock {\em Invent. Math.}, 172(1):117--168, 2008.

\bibitem[J{\o}r09]{Jorgensen-Wakamatsu}
Peter J{\o}rgensen.
\newblock Auslander-{R}eiten triangles in subcategories.
\newblock {\em J. K-Theory}, 3(3):583--601, 2009.

\bibitem[Kel05]{Keller-Orbit}
Bernhard Keller.
\newblock On triangulated orbit categories.
\newblock {\em Doc. Math.}, 10:551--581, 2005.

\bibitem[KR07]{KR1}
Bernhard Keller and Idun Reiten.
\newblock Cluster-tilted algebras are {G}orenstein and stably {C}alabi-{Y}au.
\newblock {\em Adv. Math.}, 211(1):123--151, 2007.

\bibitem[ML98]{MacLane}
Saunders Mac~Lane.
\newblock {\em Categories for the working mathematician}, volume~5 of {\em
  Graduate Texts in Mathematics}.
\newblock Springer-Verlag, New York, second edition, 1998.

\bibitem[Nak13]{Nakaoka-Twin}
Hiroyuki Nakaoka.
\newblock General heart construction for twin torsion pairs on triangulated
  categories.
\newblock {\em J. Algebra}, 374:195--215, 2013.

\bibitem[Pla11]{Plamondon-CC}
Pierre-Guy Plamondon.
\newblock Cluster characters for cluster categories with infinite-dimensional morphism spaces.
\newblock {\em Adv. Math.}, 227(1):1--39, 2011.

\bibitem[RvdB02]{rvdb}
I. Reiten, M. Van den Bergh.
\newblock Noetherian hereditary abelian categories satisfying Serre duality. 
\newblock {\em J. Amer. Math. Soc.}, 15(2), 295--366,
2002.

\bibitem[Yan12]{Yang-ClusterTube}
Dong Yang.
\newblock Endomorphism algebras of maximal rigid objects in cluster tubes.
\newblock {\em Comm. Algebra}, 40(12):4347--4371, 2012.

\end{thebibliography}
\end{document}